\documentclass[11pt]{article}
\usepackage{amsmath}
\usepackage{mathrsfs}
\usepackage{amsthm}
\usepackage[xetex]{graphicx}
\usepackage{epstopdf}
\usepackage{amssymb,enumerate}
 \usepackage{graphicx}
 \usepackage{latexsym}
\newtheorem{thm}{Theorem}[section]

\newtheorem{lem}{Lemma}[section]

\newtheorem{defn}{Definition}[section]

\newtheorem{hyp}{Hypothesis}[section]

\makeatletter

\newcommand{\Rmnum}[1]
{expandafter\@slowromancap\romannumeral#1@}
\makeatother
\numberwithin{equation}{section}
\numberwithin{figure}{section}
\allowdisplaybreaks[4]
\begin{document}

\title{Limiting behavior of inertial manifolds for stochastic differential equations driven by non-Gaussian L\'{e}vy noise}
\author{Longyu Wu,\,\,\,Ji Shu\thanks{Corresponding author. \textit{E-mail addresses}:  wulongyul@163.com(L.Wu),
 shuji@sicnu.edu.cn(J.Shu).}\small\\
\textit{School of Mathematical Sciences,  Sichuan Normal University, Chengdu 610066, PR China}\small}

\date{}
\maketitle
 {\begin{center}
\begin{minipage}{6.0in}
{\textbf{Abstract:} In this paper, we study the limiting behavior for stochastic differential equations driven by non-Gaussian $\alpha$-stable L\'{e}vy noise as $\alpha$ approaches 2. We first prove the convergence of solutions for system driven by $\alpha$-stable L\'{e}vy noise to those of the system driven by Brownian motion. Then we construct the $C^1$ inertial manifolds for both systems and show that these inertial manifolds converge in probability as $\alpha\rightarrow2$. }
\end{minipage}
\end{center}}
\section{Introduction}
 \indent
 \par
We consider a class of stochastic differential equations driven by $\alpha$-stable L\'{e}vy noise in a separable Hilbert space $H$:
\begin{equation}
\label{eq:3-1.1}
dv^{\alpha}+Av^{\alpha}=F(v^{\alpha})dt+v^{\alpha}\diamond L^{\alpha}_{t},
\end{equation}
where $ -A $ is a generator of an analytic semigroup $e^{-At}$ on $H$, $F$ is a nonlinear operator from a fractional power space $H^{\sigma}$ into $H$ for $0\leq\sigma<1$, $L^{\alpha}$ is an $\alpha$-stable L\'{e}vy process and the symbol $ \diamond$ denotes the Marcus integral.  \par

Let $ \{W_t\}_{t \geq 0} $ be a one-dimensional standard Brownian motion. Let $\{S_t^\alpha\}_{t \geq 0} $ be an independent $ \frac{\alpha}{2} $-stable subordinator and $ J^\alpha $ be a Poisson random measure with intensity measure $ dt\lambda^\alpha(dx)$, where $\lambda^\alpha$ is a L\'evy measure on $ \mathbb{R}^+ $ and
\begin{equation*}
\lambda^\alpha(dx) := \mathbb{E}J^\alpha(1, dx) = C(\alpha) \frac{dx}{|x|^{1+\frac{\alpha}{2}}},
\end{equation*}
with
\begin{equation*}
C(\alpha) := \frac{1}{2^{1 - \frac{\alpha}{2}}} \frac{\alpha}{\Gamma\left(1 - \frac{\alpha}{2}\right)}.
\end{equation*}
Since $\lim\limits_{s \to 0} s\Gamma(s) = 1 $, we get
\begin{equation*}
\lim\limits_{\alpha \to 2} \frac{2C(\alpha)}{2-\alpha} = 1.
\end{equation*}
The $ \frac{\alpha}{2}$-stable subordinator $ S^\alpha $ is given by
\begin{equation*}
S_t^\alpha = \int_0^\infty x J^\alpha(t, dx).
\end{equation*}
The subordinated Brownian motion in $ \mathbb{R} $ is defined by $ \{W_{S_t^\alpha}\}_{t \geq 0} $. It is well-known that the subordinated Brownian motion $ \{W_{S_t^\alpha}\}_{t \geq 0}$ is an $\alpha $-stable process (see \cite{Applebaum09,Sato99}). We sometimes denote
$ W_{S_t^\alpha}$ by $L^{\alpha}_{t}$.
\par
In the current paper, we consider the following stochastic differential equation driven by subordinated Brownian motion in separable Hilbert spaces $H$,
\begin{equation}
	\label{eq:3-1.2}
\frac{dv^\alpha}{dt}+Av^{\alpha}= F(v^\alpha) + v^\alpha \diamond \dot{W}_{S_t^\alpha}, \quad v^\alpha(0) = x \in H.
\end{equation}
When $\alpha = 2$, we have the standard Brownian motion, which the Marcus integral reduces to the Stratonovich stochastic integral. Then we get equation
\begin{equation}
\label{eq:3-1.3}
\frac{dv}{dt}+Av =F(v) + v \circ \dot{W}_t, \quad v(0) = x \in H,
\end{equation}
where the symbol $\circ$ denotes the Stratonovich integral and $W_t$ is a scalar standard Brownian motion. \par
Inertial manifolds are useful tools in the study of the long-time behaviour of infinite dimensional dynamical systems. The concept of inertial manifolds was introduced by Foias, Sell and Temam \cite{Foias88}. These manifolds are finite-dimensional invariant surface which attracts exponentially
all solutions, and thus the long-time dynamics restricted to the inertial manifold reduces to a finite system,
which is called an inertial form. The theory of inertial manifolds has
been extensively studied for deterministic and stochastic systems, as exemplified by reference such as \cite{Bensoussan95,Boutet98,Chow92,Chueshov95,Chueshov01,Constantin89,Da96,Henry81,Koksch02,Temam88,Zhao22,Zhao24}. These studies are focused on either deterministic differential equations or stochastic differential equations driven by Gaussian noise. \par
However, non-Gaussian noise plays an important role in many complex systems, which are often observed in biological, physical and
neural systems, etc. In particular, non-Gaussian L\'{e}vy noise has been widely applied 
in recent researches, see for example \cite{Applebaum09,Protter05,Peszat07,Qiao22,Sato99,Xu24} and the references therein. As a special non-Gaussian stochastic process, the $\alpha$-stable L\'{e}vy process are gaining increasing attention in applications. Gu and Ai \cite{Gu14} proved the existence of a random attractor for stochastic lattice
dynamical systems with $\alpha$-stable L\'{e}vy noises. Qiao and Duan \cite{Qiao15} studied the asymptotic methods for stochastic dynamical systems with small non-Gaussian L\'{e}vy noise. Yuan et al. \cite{Yuan19} constructed the slow manifolds for dynamical systems with non-Gaussian $\alpha$-stable L\'evy noise. Chao et al. \cite{Chao20} explored the role of slow manifolds in parameter estimation for a multiscale stochastic system with $\alpha$-stable L\'{e}vy noise. Liu \cite{Liu21} discussed the random stable invariant manifolds of stochastic evolution equations driven by Gaussian and non-Gaussian noises. Liu \cite{Liu22} studied the convergence behavior of the solutions to stochastic differential equations with Lipschitz and H\"{o}lder drifts, and driven by $\alpha$-stable L\'{e}vy processes. To the best of our knowledge, there are few results available in the literature concerning invariant manifolds of stochastic differential equations driven by non-Gaussian L\'{e}vy noise. In this work, we investigate the $C^1$-convergence of inertial manifolds for the system subject to an $\alpha$-stable L\'{e}vy process. \par

The goal of this paper is to construct the $C^1$ inertial manifolds for stochastic differential equations driven by $\alpha$-stable L\'{e}vy noise, as well as to study the convergence both in the solutions and $C^1$ inertial manifolds of systems driven by Gaussian and non-Gaussian noises. More precisely, we first prove the convergence relation of solutions between equation \eqref{eq:3-1.2} and equation \eqref{eq:3-1.3}. Then, we establish the existence of $C^1$ inertial manifolds for equation \eqref{eq:3-1.2} and equation \eqref{eq:3-1.3}. Furthermore, we show the $C^1$ inertial manifolds of equation \eqref{eq:3-1.2} converges in probability to those of equation \eqref{eq:3-1.3} as $\alpha \rightarrow 2$. \par
We organize this paper as follows. In Section 2, we first introduce some basic concepts and fundamental assumptions. In Section 3, we prove the convergence of solutions for equation \eqref{eq:3-1.2} and equation \eqref{eq:3-1.3}. In Section 4, we construct the $C^1$ inertial manifolds and show the convergence of these inertial manifolds for equation \eqref{eq:3-1.2} and equation \eqref{eq:3-1.3}.

\section{Preliminaries}
\indent
\par
Let $(\Omega, \mathcal{F}, \mathbb{P})$ be a probability space and $H$ be a separable Hilbert space. We denote by
$\mathcal{B}(\mathbb{R})$, $\mathcal{B}(\mathbb{R}^{+})$ and $\mathcal{B}(H)$ the collections of Borel sets on $\mathbb{R}$, $\mathbb{R}^{+}$ and $ H $, respectively. We first recall some basic concepts about random dynamical systems which from Arnold \cite{Arnold98}. \par
\subsection{Random dynamical system}
\indent
\par
\begin{defn}\label{Def:2.1}
$(\Omega, \mathcal{F}, \mathbb{P}, \{\theta_{t}\}_{t\in \mathbb{R}})$ is called a metric dynamical system if\\

\noindent
$(\mathbf{i})$\ $\theta:\mathbb{R}\times\Omega\rightarrow \Omega$ is $(\mathcal{B}(\mathbb{R})\otimes \mathcal{F}, \mathcal{F})$-measurable;\\
$(\mathbf{ii})$\ $\theta_{0}=id_{\Omega} $, the identity on $\Omega, \theta_{t+s}=\theta_{t}\circ\theta_{s}$ for all $ t,s\in \mathbb{R} $;\\
$(\mathbf{iii})$\ $\theta_{t}\mathbb{P}=\mathbb{P}$ for all $ t\in\mathbb{R} $.
\end{defn}
\begin{defn}\label{Def:2.2}
A mapping\par
\begin{center}
 $\phi:\mathbb{R}^{+}\times\Omega\times H\rightarrow H,\  (t, \omega, x)\mapsto\phi(t, \omega, x) $
\end{center}
is called a random dynamical system over a metric dynamical system $(\Omega, \mathcal{F}, \mathbb{P}, \{\theta_{t}\}_{t\in \mathbb{R}})$ if \\

\noindent
$(\mathbf{i})$\ $ \phi $ is $(\mathcal{B}(\mathbb{R}^{+})\otimes \mathcal{F}\otimes\mathcal{B}(H),  \mathcal{B}(H))$-measurable;\\
$(\mathbf{ii})$\ the mapping $ \phi(t,\omega) :=\phi(t, \omega, \cdot): H\rightarrow H $ forms a cocycle over $ \theta_{t} $:\\
\begin{center}
$ \phi(0, \omega)= id_{H} $, \ $\forall \omega\in\Omega$,
\end{center}
\begin{center}
$ \phi(t+s, \omega)= \phi(t, \theta_{s}\omega)\circ\phi(s, \omega) $, \ $\forall s, t\in\mathbb{R}^{+},\ \omega\in\Omega $.
\end{center}
$\phi$ is called a $C^{1}$ smooth random dynamical system if $\phi$ is a random dynamical system and for each $(t,\omega)\in \mathbb{R}^{+}\times\Omega$ the mapping
\begin{center}
$\phi(t,\omega): H\rightarrow H,\ x\mapsto\phi(t, \omega)x$
\end{center}
is $C^{1}$.
\end{defn}
\par
\subsection{Random inertial manifolds}
\begin{defn}\label{Def:2.3}
 A multifunction $ M=(M(\omega))_{\omega\in\Omega} $ of nonempty closed sets $ M(\omega)\subset H $, $\omega\in \Omega$, is called a random set if
\begin{align*}
\omega\mapsto\inf_{y\in M(\omega)}\|x-y\|
\end{align*}
is a random variable for all $ x\in H$.
\end{defn}
\begin{defn}\label{Def:2.4}
A random set $ M(\omega) $ is called an invariant set for a random dynamical system $ \phi(t,\omega,x) $ if
\begin{align*}
\phi(t,\omega,M(\omega))\subset M(\theta_{t}\omega), \ \forall t\geq 0.
\end{align*}
\end{defn}
\begin{defn}\label{Def:2.5}
A invariant set $ M=(M(\omega))_{\omega\in\Omega} \subset H$ is called Lipschitz (resp. $C^1$) inertial manifold of random dynamical system $ \phi(t,\omega,x) $ if the following properties are satisfied:\\
$(\mathbf{i})$\ $M(\omega)$ can be represented by a graph of a Lipschitz (resp. $C^1$) map $m(\omega,\cdot): H_1 \to H_2$, i.e.,
\begin{align*}
  M(\omega) = \{ \xi + m(\omega,\xi) \mid \xi \in H_1 \},
\end{align*}
where $H_1$ is a finite dimensional subspace of $H$ such that $H=H_1\oplus H_2$.\\
$(\mathbf{ii})$\ $M$ is exponentially attracting: there exists a positive constant $b$ such that for every $x \in H$, there is a random variable $K = K(\omega, x)$ such that
  \begin{align*}
  \text{dist}(\varphi(t, \omega, x), M(\theta_t \omega)) \leq K e^{-b t}, \  \forall t \geq 0.
  \end{align*}
\end{defn}
We say that $M$ has the \emph{asymptotic completeness} property: For any $u_0 \in H$, there exists $\tilde{u}_0 = \tilde{u}_0(\omega) \in M(\omega)$ and $C$, $b > 0$ such that
\begin{align*}
\|\varphi(t, \omega, u_0) - \varphi(t, \omega, \tilde{u}_0)\| \leq C e^{-b t} \|u_0 - \tilde{u}_0\|, \  \forall t \geq 0.
\end{align*}
Clearly, $M$ has the asymptotic completeness property implies that $M$ is exponentially attracting.\par

\subsection{Setting}
\indent
\par
Let $H$ be a separable Hilbert spaces with norm $\parallel\cdot\parallel$, we will specify the assumptions on the linear operator $A$ and nonlinear term $F$.
\begin{hyp}
\label{hyp:2.1}
$A$ is a positive operator with discrete spectrum on $H$, i.e., there exists an orthonormal basis $\{e_k\}$ of $H$ such that
\begin{align*}
A e_k = \lambda_k e_k, \ with\  0 < \lambda_1 \leq \lambda_2 \leq \cdots, \ \lim_{k \to +\infty} \lambda_k = +\infty.
\end{align*}
\end{hyp}
For arbitrary $\sigma \geq 0$, we denote by $D(A^{\sigma})$ the domain of $A^{\sigma}$, which is a Banach space under the norm $\|\cdot\|_{\sigma} = \|A^{\sigma} \cdot\|$. Let $0\leq\sigma <1$, we make the following hypothesis on $F$.
\begin{hyp}
\label{hyp:2.2}
$F : D(A^{\sigma}) \to H$ satisfies the Lipschitz condition, i.e., there exists a constant $L>0$ such that
\begin{align}
\label{3-2.1}
\|F(u_1) - F(u_2)\| \leq L \|u_1 - u_2\|_{\sigma}, \ \forall u_1, u_2 \in D(A^{\sigma}),
\end{align}
and $F(0) = 0$.
\end{hyp}
Fixing an integer $N \geq 1$ such that $\lambda_N < \lambda_{N+1}$ and denote by $P = P_N$ the orthogonal projector onto the space spanned by the first $N$ eigenvectors of $A$. Let $Q = I - P$, then one has the following estimates (see \cite{Chueshov01}):

\begin{equation}
\begin{aligned}
\label{3-2.2}
\| A^{\sigma} e^{-At} P \| &\leq \lambda_N^{\sigma} e^{\lambda_N |t|}, \  \forall t \in \mathbb{R},\\
\| e^{-At} Q \| &\leq e^{-\lambda_{N+1} t}, \  \forall t \geq 0,\\
\| A^{\sigma} e^{-At} Q \| &\leq \left[ \left( \frac{\sigma}{t} \right)^{\sigma} + \lambda_{N+1}^{\sigma} \right] e^{-\lambda_{N+1} t}, \  \forall t > 0, \ \forall \sigma > 0.
\end{aligned}
\end{equation}
\par
Let $D([0,T], \mathbb{R})$ be the space of c\`adl\`ag $\mathbb{R}$-valued functions on $[0,T]$. The following $J_1$-metric was defined by Skorokhod \cite{Skorokhod56}:
\begin{equation*}
d_{J_1}(\varphi_1, \varphi_2) = \inf_{\lambda \in \Lambda} \left\{ \sup_{0 \leq t \leq T} |\varphi_1(t) - \varphi_2(\lambda(t))| + \sup_{\substack{s,t \in [0,T]\\ s \neq t}} \left| \log \frac{\lambda(s) - \lambda(t)}{s-t} \right| \right\}, \  \varphi_1, \varphi_1 \in D([0,T], \mathbb{R}),
\end{equation*}
where $\Lambda$ is the set of all the strictly increasing continuous functions that map $[0,T]$ onto itself.
Let $D_{J_1}([0,T], \mathbb{R})$ be $D([0,T], \mathbb{R})$ equipped with the $J_1$-metric, and $D_U([0,T], \mathbb{R})$ denote $D([0,T], \mathbb{R})$ equipped with the usual uniform metric $d_U$, see \cite{Billingsley99, Jacod03, Pollard12}.\par
In this paper, we consider the two-sided subordinated Brownian motion $ W_{S_t^\alpha}$. Let $\mathbb{W}$ be the space of all continuous functions from $\mathbb{R}$ to $\mathbb{R}$ taking zero value at $t=0$, which is endowed the compact open (locally uniform
convergence) topology and Wiener measure $\mu_{\mathbb{W}}$ so that the coordinate process $ W_{t}(\omega)=\omega_{t}$ is a standard one-dimensional Brownian motion. Let $\mathbb{S}$ be the space of all increasing and c\`adl\`ag functions from $\mathbb{R}$ to $\mathbb{R}$ with $\lim\limits_{s \to 0} \ell_s = 0$, which is endowed with the $J_1$-Skorokhod metric and the probability $\mu_{\mathbb{S}}$ so that the coordinate process
\begin{align*}
S_t^\alpha(\ell^\alpha) := \ell_t^\alpha
\end{align*}
is an $\frac{\alpha}{2}$-stable subordinator $S_t^\alpha$ (see \cite{Sato99, Zhang13}).
Consider the following product probability space:
\begin{align*}
(\Omega, \mathcal{F}, \mathbb{P}) := (\mathbb{W} \times \mathbb{S}, \mathcal{B}(\mathbb{W}) \times \mathcal{B}(\mathbb{S}), \mu_{\mathbb{W}} \times \mu_{\mathbb{S}})
\end{align*}
and define
\begin{align*}
W_{S_t^\alpha} = L_t^\alpha(\omega, \ell^\alpha) := \omega_{\ell_t^\alpha}.
\end{align*}
Then $\{L_t^\alpha\}_{t \geq 0}$ is an $\alpha$-stable process on $(\Omega, \mathcal{F}, \mathbb{P})$. Since $(\mathbb{W}, \mathcal{B}(\mathbb{W}), \mu_{\mathbb{W}})$ and $(\mathbb{S}, \mathcal{B}(\mathbb{S}), \mu_{\mathbb{S}})$ are two Polish spaces, then $(\Omega, \mathcal{F}, \mathbb{P})$ is also a Polish space. The driving system $\theta$ on $\Omega$ is defined by the shift operator
\begin{align*}
\theta_t(\omega_{\ell_s^\alpha}) = \omega_{\ell_{t+s}^\alpha}-\omega_{\ell_t^\alpha},
\end{align*}
which satisfies the condition in Definition \ref{Def:2.1}.
Then the map $(t, \omega_{\ell^\alpha}) \mapsto \theta_t \omega_{\ell^\alpha}$ is continuous, thus measurable, and the L\'evy probability measure $\mathbb{P}$ is $\theta$-invariant, see \cite{Liu21}.

\section{Convergence of solutions}
\indent
\par
In this section, we show the solutions of equations \eqref{eq:3-1.2} converge to the solutions of equation \eqref{eq:3-1.3} as $\alpha \rightarrow 2$. Firstly, we give the conjugate random equations of equation \eqref{eq:3-1.2} and equation \eqref{eq:3-1.3}, respectively.
Let $W_t$ be a two-sided standard Brownian motion, then $W_{S_t^\alpha} = L_t^\alpha(\omega, \ell^\alpha)$ is a two-sided symmetric $\alpha$-stable L\'evy motion on $\mathbb{R}$. We consider the Langevin equation
\begin{align}
\label{3-2.3}
dz = -zdt + dW_{S_t^\alpha}.
\end{align}
A solution of this equation is called an Ornstein--Uhlenbeck process. Moreover, for every $\alpha \in (1,2)$, we have the following lemmas from \cite{Liu21}.
\begin{lem}
\label{lem:3-2.1}
$(\mathbf{i})$\ There exists a $\{\theta_t: t \in \mathbb{R}\}$-invariant set (still denoted as) $\Omega$ of full measure such that the sample paths $\omega_{\ell^\alpha_t}$ of $W_{S_t^\alpha}$ satisfy
\begin{align*}
\lim_{t \to \pm \infty} \frac{\omega_{\ell^\alpha_t}}{t} = 0, \quad \omega \in \Omega.
\end{align*}
$(\mathbf{ii})$\
The random variable
\begin{align*}
z(\omega_{\ell^\alpha}) = - \int_{-\infty}^{0} e^{s} \omega_{\ell^\alpha_s} ds, \quad \omega \in \Omega
\end{align*}
is well-defined and the unique stationary solution of \eqref{3-2.3} is given by
\begin{align*}
z(\theta_t \omega_{\ell^\alpha}) &= - \int_{-\infty}^{0} e^{s} \theta_t \omega_{\ell^\alpha_s} ds = \omega_{\ell^\alpha_t} - \int_{-\infty}^{0} e^{s} \omega_{\ell^\alpha_{t+s}} ds.
\end{align*}
Moreover, the mapping $t \mapsto z(\theta_t \omega_{\ell^\alpha})$ is c\`adl\`ag.\\
$(\mathbf{iii})$\ In addition,
\begin{align*}
\lim_{t \to \pm \infty} \frac{|z(\theta_t \omega_{\ell^\alpha})|}{|t|} = 0, \  \mathrm{and} \  \lim_{t \to \pm \infty} \frac{1}{t} \int_0^t z(\theta_s \omega_{\ell^\alpha}) ds = 0.
\end{align*}
\end{lem}
We now replace $\mathcal{B}(\mathbb{W}) \times \mathcal{B}(\mathbb{S})$ by
\begin{align*}
\mathcal{F} = \{ \Omega \cap B : B \in \mathcal{B}(\mathbb{W}) \times \mathcal{B}(\mathbb{S}) \}
\end{align*}
for $\Omega$ given in Lemma \ref{lem:3-2.1}. The probability measure is the restriction of $\mu_{\mathbb{W}} \times \mu_{\mathbb{S}}$ to this new $\sigma$-algebra, also denoted by $\mathbb{P}$. In the following, we will consider the metric dynamical system $(\Omega, \mathcal{F}, \mathbb{P}, \mathbb{R}, \theta)$. Note that for standard Brownian motion $W_t$, we have the similar results, see \cite{Duan03}.\par
\begin{lem}
\label{lem:3-2.2}
For every $p$ in $(0,2)$ and $T > 0$,
\begin{align*}
\lim_{\alpha \to 2} \mathbb{E} \sup_{-T \leq t \leq T} \left| z(\theta_t \omega_{\ell^\alpha}) - z(\theta_t \omega) \right|^p = 0.
\end{align*}
\end{lem}
\noindent This lemma characterizes the convergence behavior of the stationary Ornstein--Uhlenbeck process $z(\theta_t \omega_{\ell^\alpha})$ as $\alpha \rightarrow 2$.\par
Let $u^\alpha=e^{-z(\theta_t \omega_{\ell^\alpha})}v^\alpha$. Then, $u^\alpha$ satisfies the following random differential equation
\begin{equation}
\label{eq:3-2.4}
\frac{du^\alpha}{dt} + Au^\alpha = z(\theta_t \omega_{\ell^\alpha}) u^\alpha + G(\theta_t \omega_{\ell^\alpha}, u^\alpha),
\end{equation}
where $G(\omega_{\ell^\alpha}, u^\alpha) := e^{-z(\omega_{\ell^\alpha})} F(e^{z(\omega_{\ell^\alpha})} u^\alpha)$. We denote $u^\alpha(t, \omega_{\ell^\alpha}, x)$ by the solution of \eqref{eq:3-2.4} with initial condition $u^\alpha(0, \omega_{\ell^\alpha}, x)=x$. It is easy to verify that $G$ has the same Lipschitz constant $L$ as $F$. Obviously, the solution mapping $(t, \omega_{\ell^\alpha}, x) \mapsto u^\alpha(t, \omega_{\ell^\alpha}, x)$ generates a random dynamical system.\par
Similarly, let $u=e^{-z(\theta_t \omega)}v$, we get random differential equation
\begin{equation}
\label{eq:3-2.5}
\frac{du}{dt} + Au = z(\theta_t \omega) u + G(\theta_t \omega, u),
\end{equation}
where $G(\omega, u) := e^{-z(\omega)} F(e^{z(\omega)} u)$. The same conclusions as for equation \eqref{eq:3-2.4} also hold for equation \eqref{eq:3-2.5}.\par
For $x \in H$ and $\omega_{\ell^\alpha} \in \Omega$, we introduce the transformation
\begin{align*}
T(\omega_{\ell^\alpha}, x) = e^{-z(\omega_{\ell^\alpha})}x,
\end{align*}
and its inverse transformation
\begin{align*}
T^{-1}(\omega_{\ell^\alpha}, x) = e^{z(\omega_{\ell^\alpha})}x.
\end{align*}
Observe that
\begin{align*}
v^\alpha(t, \omega_{\ell^\alpha}, x)=T^{-1}(\theta_{t}\omega_{\ell^\alpha}, \cdot)\circ u^\alpha(t, \omega_{\ell^\alpha}, T(\omega_{\ell^\alpha}, x)),
\end{align*}
from which we find that the solution mapping $ (t, \omega_{\ell^\alpha}, x)\mapsto v^\alpha(t, \omega_{\ell^\alpha}, x)$ also generate a random dynamical system.\par
Recall that $-A$  is a generator of analytic semigroup $e^{-At}$. There exists constants $\kappa>0$ and $M\geq1$ such that
\begin{align}
\label{3-3.4}
\| A^{\sigma} e^{-At} \| \leq M t^{-\sigma} e^{-\kappa t}, \  \forall t> 0,
\end{align}
see \cite{Pazy83}.\par
Next we give the estimate of $u(t, \omega, x)$.
\begin{lem}
\label{lem:3-3.3}
For every $\omega \in \mathbb{W} \subset\Omega$ and $x \in H^{\sigma}$, there exists a positive constant $C$ depending on $T$, $\sigma$ and $x$ such that
\begin{align}
\label{3-3.5}
\sup_{t \in [0,T]} \|u(t,\omega,x)\|_{\sigma} \leq C.
\end{align}
\end{lem}
\begin{proof}
Following \eqref{eq:3-2.5}, we have
\begin{align*}
u(t,\omega,x) = e^{-A t + \int_0^t z(\theta_r \omega) dr} x + \int_0^t e^{-A(t-s)+\int_s^t z(\theta_r \omega) dr} G(\theta_s \omega, u(s,\omega,x)) ds.
\end{align*}
By Hypothesis \ref{hyp:2.2} and \eqref{3-3.4}, for all $t \in [0,T]$, we get
\begin{align*}
\|u(t,\omega,x)\|_{\sigma} \leq M e^{-\kappa t} N_{\omega,T} \|x\|_{\sigma} + M L N_{\omega,T} \int_0^t (t - s)^{-\sigma} e^{-\kappa(t-s)} \|u(s)\|_{\sigma} ds,
\end{align*}
where $ N_{\omega,T}= \int_0^T |z(\theta_r \omega)| dr$. Let
\begin{align*}
E_{\sigma}(x) &= \sum_{n=0}^{+\infty} x^{n(1-\sigma)}/\Gamma(n(1-\sigma) + 1), \  \text{for } x \geq 0,\\
\theta&= \left[  M L N_{\omega,T}\Gamma(1 - \sigma) \right]^{1 / (1 - \sigma)}.
\end{align*}
Here $\Gamma(\cdot)$ is the Gamma function. It then follows from \cite[Lemma 7.1.1]{Henry81} that
\begin{align*}
e^{\kappa t} \|u(t,\omega,x)\|_{\sigma} \leq e^{\kappa T} M N_{\omega,T} \|x\|_{\sigma} E_{\sigma}(\theta t), \  \forall t \in [0,T].
\end{align*}
Thus \eqref{3-3.5} holds. This completes the proof.
\end{proof}

\begin{lem}
\label{lem:3-3.4}
For every $\omega_{\ell^\alpha} \in \Omega$, $x \in H^{\sigma}$ and $T>0$, we have
\begin{align}
\label{3-3.6}
\sup_{t \in [0,T]} \|u^\alpha(t, \omega_{\ell^\alpha}, x)-u(t,\omega,x)\|_{\sigma} \rightarrow 0
\end{align}
in probability as $\alpha\rightarrow2$.
\end{lem}
\begin{proof}
Let $\tilde{u}^\alpha(t)=u^\alpha(t)-u(t)$, by \eqref{eq:3-2.4} and \eqref{eq:3-2.5}, we get
\begin{align*}
\tilde{u}^\alpha(t) =& \int_0^t e^{-A(t-s)+\int_s^t z(\theta_r \omega) dr} \Big[G(\theta_{s}\omega_{\ell^\alpha}, u^{\alpha}(s))-G(\theta_{s}\omega,u(s))\\&\ \ \  +\big(z(\theta_{s}\omega_{\ell^\alpha})-z(\theta_{s}\omega)\big)u(s)\Big]ds.
\end{align*}
It follows from \eqref{3-2.1} that
\begin{align*}
&\|G(\theta_{s}\omega_{\ell^\alpha}, u^{\alpha}(s))-G(\theta_{s}\omega,u(s))\|\leq L\|\tilde{u}^\alpha(s)\|_{\sigma}+2L|e^{z(\theta_{s}\omega_{\ell^\alpha})-z(\theta_{s}\omega)}-1| \|u(s)\|_{\sigma}.
\end{align*}
Along with \eqref{3-3.4} and Lemma \ref{lem:3-3.3}, we have
\begin{align*}
\|\tilde{u}^\alpha(t)\|_{\sigma} \leq&  M L \int_0^t (t - s)^{-\sigma} e^{-\kappa(t-s)}e^{\int_s^t z(\theta_r \omega) dr} \|\tilde{u}^\alpha(s)\|_{\sigma} ds\\
&+C M \int_0^t (t - s)^{-\sigma} e^{-\kappa(t-s)}e^{\int_s^t z(\theta_r \omega) dr} \big[2L|e^{z(\theta_{s}\omega_{\ell^\alpha})-z(\theta_{s}\omega)}-1| +|z(\theta_{s}\omega_{\ell^\alpha})-z(\theta_{s}\omega)|\big]ds.
\end{align*}
Applying \cite[Lemma 7.1.1]{Henry81}, we obtain that
\begin{align*}
 \|\tilde{u}^\alpha(t)\|_{\sigma} \leq&  CM N_{\omega_{\ell^\alpha},T} E_{\sigma}(\tilde{\theta} t)\\
&\times\int_0^t (t - s)^{-\sigma} e^{-\kappa(t-s)} \big[2L|e^{z(\theta_{s}\omega_{\ell^\alpha})-z(\theta_{s}\omega)}-1| +|z(\theta_{s}\omega_{\ell^\alpha})-z(\theta_{s}\omega)|\big]ds,
\end{align*}
where $ N_{\omega_{\ell^\alpha},T}= \int_0^T |z(\theta_r \omega_{\ell^\alpha})| dr$ and
$\tilde{\theta}= \left[  M L N_{\omega_{\ell^\alpha},T}\Gamma(1 - \sigma) \right]^{1 / (1 - \sigma)}$.\par
 By Lemma \ref{lem:3-2.2}, for each subsequence $\{\alpha_n\}_{n \in \mathbb{N}}$, there exists a sub-subsequence $\{\alpha_{n(k)}\}_{k \in \mathbb{N}}$ such that $z(\theta_s \omega_{\ell^{\alpha_{n(k)}}})$ converges to $z(\theta_s \omega)$ in $D_u([0, T], \mathbb{R})$, almost surely, i.e., for every $T > 0$,
\begin{align*}
\lim_{k \to \infty} \sup_{s \in [0, T]} | z(\theta_s \omega_{\ell^{\alpha_{n(k)}}}) - z(\theta_s \omega) | = 0,
\end{align*}
which implies that for every $\varepsilon > 0$, there exist $K>0$ such that for every $k\geq K$,
\begin{align*}
\sup_{s \in [0, T]}| z(\theta_s \omega_{\ell^{\alpha_{n(k)}}}) - z(\theta_s \omega) | \leq \varepsilon.
\end{align*}
This yields that
\begin{align*}
\sup_{t \in [0, T]} \|\tilde{u}^{\alpha_{n(k)}}(t)\|_{\sigma} \leq CM N_{\omega_{\ell^{\alpha_{n(k)}}},T} E_{\sigma}(\tilde{\theta} t)\int_0^t (t - s)^{-\sigma} e^{-\kappa(t-s)} \big(2L|e^{\varepsilon}-1| +\varepsilon\big)ds.
\end{align*}
Therefore,
\begin{align*}
\sup_{t \in [0, T]} \|u^{\alpha_{n(k)}}(t)-u(t)\|_{\sigma} \rightarrow 0, \ \text{almost surely}.
\end{align*}
This implies that
\begin{align*}
\sup_{t \in [0, T]} \|u^{\alpha}(t)-u(t)\|_{\sigma} \rightarrow 0
\end{align*}
in probability as $\alpha\rightarrow2$. The proof is complete.
\end{proof}
\begin{thm}
\label{th:3-3.1}
Assume Hypothesis \ref{hyp:2.1} and  \ref{hyp:2.2} hold. For every $\omega_{\ell^\alpha} \in \Omega$, $x \in H^{\sigma}$ and $T>0$, we have
\begin{align}
\label{3-3.7}
\sup_{t \in [0,T]} \|v^\alpha(t, \omega_{\ell^\alpha}, x)-v(t,\omega,x)\|_{\sigma} \rightarrow 0
\end{align}
in probability as $\alpha\rightarrow2$ and in $D_U([0,T], H^{\sigma})$ as well as $D_{J_1}([0,T], H^{\sigma})$.
\end{thm}
\begin{proof}
As in Lemma \ref{lem:3-3.4}, for each subsequence $\{\alpha_n\}_{n \in \mathbb{N}}$, there exists a sub-subsequence $\{\alpha_{n(k)}\}_{k \in \mathbb{N}}$ such that
\begin{align*}
\sup_{t \in [0, T]} | z(\theta_t \omega_{\ell^{\alpha_{n(k)}}}) - z(\theta_t \omega) | \rightarrow 0
\end{align*}
and
\begin{align*}
\sup_{t \in [0, T]} \|u^{\alpha_{n(k)}}(t, \omega_{\ell^{\alpha_{n(k)}}}, x)-u(t, \omega, x)\|_{\sigma} \rightarrow 0, \ \text{almost surely}.
\end{align*}
Also we can get
\begin{align*}
\sup_{t \in [0, T]} | e^{z(\theta_t \omega_{\ell^{\alpha_{n(k)}}})} - e^{z(\theta_t \omega)} | \rightarrow 0, \ \text{almost surely}.
\end{align*}
Notice that
\begin{align*}
 &\|v^{\alpha_{n(k)}}(t, \omega_{\ell^{\alpha_{n(k)}}}, x)-v(t, \omega, x)\|_{\sigma} \\&= \|e^{z(\theta_t \omega_{\ell^{\alpha_{n(k)}}})}u^{\alpha_{n(k)}}(t, \omega_{\ell^{\alpha_{n(k)}}}, e^{-z(\omega_{\ell^{\alpha_{n(k)}}})}x)-e^{z(\theta_t \omega)}u(t, \omega, e^{-z(\theta_t \omega)}x)\|_{\sigma}\\&
 \leq | e^{z(\theta_t \omega_{\ell^{\alpha_{n(k)}}})} | \Big[\|u^{\alpha_{n(k)}}(t, \omega_{\ell^{\alpha_{n(k)}}}, e^{-z(\omega_{\ell^{\alpha_{n(k)}}})}x)-u(t, \omega, e^{-z(\omega_{\ell^{\alpha_{n(k)}}})}x)\|_{\sigma}\\ &\ \ \ \ \ +|\|u(t, \omega, e^{-z(\omega_{\ell^{\alpha_{n(k)}}})}x)-u(t, \omega, e^{-z(\omega)}x)\|_{\sigma}\Big]\\&
 \ \ \ +| e^{z(\theta_t \omega_{\ell^{\alpha_{n(k)}}})} - e^{z(\theta_t \omega)}| \|u(t, \omega, e^{-z(\omega)}x)\|_{\sigma},
\end{align*}
which along with Lemma \ref{lem:3-3.3} and \ref{lem:3-3.4}, we have
\begin{align*}
\sup_{t \in [0, T]} \|v^{\alpha_{n(k)}}(t, \omega_{\ell^{\alpha_{n(k)}}}, x)-v(t, \omega, x)\|_{\sigma} \rightarrow 0, \ \text{almost surely}
\end{align*}
in $D_U([0,T], H^{\sigma})$. Therefore,
\begin{align*}
\sup_{t \in [0,T]} \|v^\alpha(t, \omega_{\ell^\alpha}, x)-v(t,\omega,x)\|_{\sigma} \rightarrow 0
\end{align*}
in probability as $\alpha\rightarrow2$ and in $D_U([0,T], H^{\sigma})$.
\end{proof}

\section{Convergence of inertial manifolds}
\indent
\par
In this section, we first prove the existence and $C^{1}$-smoothness of inertial manifolds for equation \eqref{eq:3-2.4} and equation \eqref{eq:3-2.5}. Then we show the convergence of $C^{1}$ inertial manifolds for equation \eqref{eq:3-2.4} and equation \eqref{eq:3-2.5} as $\alpha\rightarrow 2$.\par

Let $\mathcal{C}_{\beta}^{\alpha,-}$ be a Banach space defined by
\begin{align*}
\mathcal{C}_{\beta}^{\alpha,-} = \Big\{ \varphi\in C((-\infty, 0], D(A^{\sigma})) \ \big| \sup_{t\leq0} \big\{e^{\beta t - \int_0^t z(\theta_r \omega_{\ell^\alpha}) dr} \|\varphi(t)\|_{\sigma} \big\} < +\infty \Big\}
\end{align*}
with the norm
\begin{align*}
\|\varphi\|_{\mathcal{C}_{\beta}^{\alpha,-}} = \sup_{t\leq0} \big\{e^{\beta t - \int_0^t z(\theta_r \omega_{\ell^\alpha}) dr} \|\varphi(t)\|_{\sigma}\big\}.
\end{align*}
Here $\beta:=\lambda_N+\frac{2}{\mu}L\lambda_N^{\sigma}$, $\mu \in (0,1)$, which is in $(\lambda_N,\lambda_{N+1})$. We denote the set
\begin{align*}
\mathscr{M}^{\alpha}(\omega_{\ell^\alpha})= \big\{ x \in D(A^{\sigma}) \ \big| \ u^\alpha(t, \omega_{\ell^\alpha}, x) \in \mathcal{C}_{\beta}^{\alpha,-} \big\}.
\end{align*}
In order to prove that $\mathscr{M}^{\alpha}(\omega_{\ell^\alpha})$ is a inertial manifold, we first need the following lemma.
\begin{lem}
\label{lem:3-4.1}
$x \in \mathscr{M}^{\alpha}(\omega_{\ell^\alpha})$ if and only if there exists a function $u^\alpha(\cdot) \in \mathcal{C}_{\beta}^{\alpha,-}$ with $u^\alpha(0) = x$ and satisfies
\begin{align}
\label{3-4.1}
u^\alpha(t) &= e^{-A t + \int_0^t z(\theta_r \omega_{\ell^\alpha}) dr} \xi
+ \int_0^t e^{-A(t-s) + \int_s^t z(\theta_r \omega_{\ell^\alpha}) dr} P G(\theta_s \omega_{\ell^\alpha}, u^\alpha(s)) ds \notag\\
&\qquad
+ \int_{-\infty}^t e^{-A(t-s) + \int_s^t z(\theta_r \omega_{\ell^\alpha}) dr} Q G(\theta_s \omega_{\ell^\alpha}, u^\alpha(s)) ds,
\end{align}
where $\xi = P x$.
\end{lem}

\begin{proof}
Let $x \in \mathscr{M}^{\alpha}(\omega_{\ell^\alpha})$, by the variation of constants formula, for $\tau > t$, we have that
\begin{align*}
P u^\alpha(t, \omega_{\ell^\alpha}, x) =& e^{-A t + \int_0^t z(\theta_r \omega_{\ell^\alpha}) dr} P x + \int_0^t e^{-A(t-s) + \int_s^t z(\theta_r \omega_{\ell^\alpha}) dr} P G(\theta_s \omega_{\ell^\alpha}, u^\alpha(s)) ds,
\end{align*}
and
\begin{align*}
Q u^\alpha(t, \omega_{\ell^\alpha}, x) =& e^{-A(t-\tau) + \int_s^t z(\theta_r \omega_{\ell^\alpha}) dr} Q u^\alpha(\tau, \omega_{\ell^\alpha}, x)
\nonumber\\&+ \int_\tau^t e^{-A(t-s) + \int_s^t z(\theta_r \omega_{\ell^\alpha}) dr} Q G(\theta_s \omega_{\ell^\alpha}, u^\alpha(s)) ds.
\end{align*}
From \eqref{3-2.2}, we get
\begin{align*}
&\|e^{-A(t-\tau)+\int_{\tau}^{t}z(\theta_{r}\omega_{\ell^\alpha})dr}Q u^\alpha(\tau, \omega_{\ell^\alpha}, x)\|_{\sigma}\\&\leq e^{-\lambda_{N+1} (t-\tau)}e^{-\beta\tau+\int_{0}^{t}z(\theta_{r}\omega_{\ell^\alpha})dr}\| u^\alpha\|_{{\mathcal{C}_{\beta}^{\alpha,-}}}\rightarrow0, \ as \ \tau\rightarrow-\infty.
\end{align*}
Then we have
\begin{align*}
Q u^\alpha(t, \omega_{\ell^\alpha}, x) = \int_{-\infty}^t e^{-A(t-s) + \int_s^t z(\theta_r \omega_{\ell^\alpha}) dr} Q G(\theta_s \omega_{\ell^\alpha}, u^\alpha(s)) ds.
\end{align*}
Thus \eqref{3-4.1} can be obtained. The converse follows by direct computation.
\end{proof}

\begin{thm}
\label{th:3-4.1}
Assume that the spectral gap condition
\begin{align}
\label{3-4.2}
\lambda_{N+1}-\lambda_{N}\geq\frac{2L}{\mu}\left( \lambda_{N}^{\sigma}+\sigma^{\sigma}\Gamma(1-\sigma)(\lambda_{N+1}-\lambda_{N})^{\sigma}+\lambda_{N+1}^{\sigma}\right)
\end{align}
holds for some $\mu \in(0,1)$. Then there exists a Lipschitz inertial manifold for equation \eqref{eq:3-2.4}, which is given by
\begin{align*}
 \mathscr{M}^{\alpha}(\omega_{\ell^\alpha})=\big\lbrace \xi+\psi^{\alpha}(\omega_{\ell^\alpha}, \xi)|\xi\in PD(A^{\sigma})\big\rbrace,
\end{align*}
where $\psi^{\alpha}(\omega_{\ell^\alpha}, \cdot): PD(A^{\sigma})\rightarrow QD(A^{\sigma})$ is Lipschitz continuous, and $\psi^{\alpha}$ is measurable in $(\omega_{\ell^\alpha}, \xi)$. Furthermore, $\widetilde{\mathscr{M}}^{\alpha}(\omega_{\ell^\alpha})=T^{-1}(\omega_{\ell^\alpha}, \mathscr{M}^{\alpha}(\omega_{\ell^\alpha}))$, i.e.,
\begin{align*}
\widetilde{\mathscr{M}}^{\alpha}(\omega_{\ell^\alpha})=\left\lbrace \xi+e^{z(\omega_{\ell^\alpha})}\psi^{\alpha}(\omega_{\ell^\alpha}, e^{-z(\omega_{\ell^\alpha})}\xi)|\xi\in PD(A^{\sigma})\right\rbrace
\end{align*}
is the Lipschitz inertial manifold of equation \eqref{eq:3-1.2}. \par
In addition, if $F\in C^{1}$, then $\mathscr{M}^{\alpha}( \omega_{\ell^\alpha})$ is a $C^{1}$ inertial manifold for equation \eqref{eq:3-2.4}, i.e., $\psi^{\alpha}(\omega_{\ell^\alpha}, \xi)$ is $C^{1}$ in $\xi$. Consequently, $\widetilde{\mathscr{M}}^{\alpha}(\omega_{\ell^\alpha})$ is a $C^{1}$ inertial manifold of equation \eqref{eq:3-1.2}.
\end{thm}
\begin{proof} We prove the theorem in five steps.\\
$\mathbf{Step\ 1}$. We first prove that equation \eqref{3-4.1} has a unique solution $u^{\alpha}=u^{\alpha}(\cdot, \omega_{\ell^\alpha}, \xi)$ in $\mathcal{C}_{\beta}^{\alpha,-}$. Denote the right-hand side of \eqref{3-4.1} by $\mathcal{Y}(u^{\alpha}, \omega_{\ell^\alpha}, \xi)$. Multiplying both side of \eqref{3-4.1} by $e^{\beta t - \int_0^t z(\theta_r \omega_{\ell^\alpha}) dr}$, it follows from \eqref{3-2.1} and \eqref{3-2.2} that
\begin{align*}
&e^{\beta t - \int_0^t z(\theta_r \omega_{\ell^\alpha}) dr}\Big\|\mathcal{Y}(u^{\alpha}, \omega_{\ell^\alpha}, \xi)\Big\|_{\sigma}\\
&\leq e^{(\beta-\lambda_{N}) t}\|\xi\|_{\sigma}+L\lambda_{N}^{\sigma}\int_t^0e^{(\beta-\lambda_{N})(t-s)}\|u^{\alpha}\|_{\mathcal{C}_{\beta}^{\alpha,-}}ds
\\&\ \ \ + L\int_{-\infty}^t \Big[ \Big( \frac{\sigma}{t-s} \Big)^{\sigma} + \lambda_{N+1}^{\sigma} \Big] e^{(\beta-\lambda_{N+1})(t-s)}\|u^{\alpha}\|_{\mathcal{C}_{\beta}^{\alpha,-}}ds \\
&\leq \|\xi\|_{\sigma}+L\left( \frac{\lambda_{N}^{\sigma}}{\beta-\lambda_{N}}+\frac{\sigma^{\sigma}\Gamma(1-\sigma)}{(\lambda_{N+1}-\beta)^{1-\sigma}}+\frac{\lambda_{N+1}^{\sigma}}{\lambda_{N+1}-\beta}\right)\|u^{\alpha}\|_{\mathcal{C}_{\beta}^{\alpha,-}},
\end{align*}
which implies that $\mathcal{Y}(u^{\alpha}, \omega_{\ell^\alpha}, \xi)$ maps $\mathcal{C}_{\beta}^{\alpha,-}$ into itself.\par
For any $u_{1}^{\alpha},u_{2}^{\alpha}\in \mathcal{C}_{\beta}^{\alpha,-} $, we have
\begin{align*}
&\Big\|\mathcal{Y}(u_{1}^{\alpha}, \omega_{\ell^\alpha}, \xi)-\mathcal{Y}(u_{2}^{\alpha}, \omega_{\ell^\alpha}, \xi)\Big\|_{\mathcal{C}_{\beta}^{\alpha,-}}\\
&\leq L\lambda_{N}^{\sigma}\int_t^0e^{(\beta-\lambda_{N})(t-s)}ds \|u_{1}^{\alpha}-u_{2}^{\alpha}\|_{\mathcal{C}_{\beta}^{\alpha,-}}
\\&\ \ \ + L\int_{-\infty}^t \Big[ \Big( \frac{\sigma}{t-s} \Big)^{\sigma} + \lambda_{N+1}^{\sigma} \Big] e^{(\beta-\lambda_{N+1})(t-s)}ds \|u_{1}^{\alpha}-u_{2}^{\alpha}\|_{\mathcal{C}_{\beta}^{\alpha,-}}\\
&\leq L\left( \frac{\lambda_{N}^{\sigma}}{\beta-\lambda_{N}}+\frac{\sigma^{\sigma}\Gamma(1-\sigma)}{(\lambda_{N+1}-\beta)^{1-\sigma}}+\frac{\lambda_{N+1}^{\sigma}}{\lambda_{N+1}-\beta}\right)\|u_{1}^{\alpha}-u_{2}^{\alpha}\|_{\mathcal{C}_{\beta}^{\alpha,-}}.
\end{align*}
Since $\lambda_{N+1}-\beta\leq\lambda_{N+1}-\lambda_{N}$, we find that
\begin{align*}
L\left( \frac{\lambda_{N}^{\sigma}}{\beta-\lambda_{N}}+\frac{\sigma^{\sigma}\Gamma(1-\sigma)}{(\lambda_{N+1}-\beta)^{1-\sigma}}+\frac{\lambda_{N+1}^{\sigma}}{\lambda_{N+1}-\beta}\right)
&\leq L\left( \frac{\lambda_{N}^{\sigma}}{\beta-\lambda_{N}}+\frac{\sigma^{\sigma}\Gamma(1-\sigma)(\lambda_{N+1}-\lambda_{N})^{\sigma}+\lambda_{N+1}^{\sigma}}{\lambda_{N+1}-\beta}\right)
\\ &\leq\mu<1.
\end{align*}
Thus, \eqref{3-4.2} implies that $\mathcal{Y}(\cdot, \omega_{\ell^\alpha}, \xi)$ is a contraction map. Applying the contraction mapping principle, the mapping $\mathcal{Y}(\cdot, \omega_{\ell^\alpha}, \xi)$ has a unique fixed point $\bar{u}^{\alpha}(\cdot, \omega_{\ell^\alpha}, \xi)\in \mathcal{C}_{\beta}^{\alpha,-} $. Then, for any $\xi_{1}, \xi_{2}\in PD(A^{\sigma})$, we get
\begin{align*}
&\|\bar{u}^{\alpha}(t, \omega_{\ell^\alpha}, \xi_{1})-\bar{u}^{\alpha}(t, \omega_{\ell^\alpha}, \xi_{2})\|_{\mathcal{C}_{\beta}^{\alpha,-}}
\\&\leq \|\xi_{1}-\xi_{2}\|+\mu\|u^{\alpha}(\cdot, \omega_{\ell^\alpha}, \xi_{1})-u^{\alpha}(\cdot, \omega_{\ell^\alpha}, \xi_{2})\|_{\mathcal{C}_{\beta}^{\alpha,-}}.
\end{align*}
Hence, we obtain
\begin{align}
\label{3-4.3}
\|\bar{u}^{\alpha}(t, \omega_{\ell^\alpha}, \xi_{1} )-\bar{u}^{\alpha}(t, \omega_{\ell^\alpha}, \xi_{2})\|_{\mathcal{C}_{\beta}^{\alpha,-}}\leq (1-\mu)^{-1}\|\xi_{1}-\xi_{2}\|_{\sigma}.
\end{align}
This implies that $\bar{u}^{\alpha}$ is Lipschitz continuous in $ \xi$.\\
$\mathbf{Step\ 2}$. We prove the existence of Lipschitz invariant manifolds for equation \eqref{eq:3-2.4}. Let $\psi^{\alpha}(\omega_{\ell^\alpha}, \xi)=Q\bar{u}^{\alpha}(0, \omega, \xi)$, it follows from \eqref{3-4.1} that
\begin{align*}
\psi^{\alpha}(\omega_{\ell^\alpha}, \xi)=                   \int_{-\infty}^0e^{As+\int_{s}^{0}z(\theta_{r}\omega_{\ell^\alpha})dr}QG(\theta_{s}\omega_{\ell^\alpha}, \bar{u}^{\alpha}(s, \omega_{\ell^\alpha}, \xi))ds.
\end{align*}
By Lemma \ref{lem:3-4.1}, we get
\begin{align*}
 \mathscr{M}^{\alpha}(\omega_{\ell^\alpha})=\big\lbrace \xi+\psi^{\alpha}(\omega_{\ell^\alpha}, \xi)|\xi\in PD(A^{\sigma})\big\rbrace.
\end{align*}
For any $\xi_{1}, \xi_{2}\in PD(A^{\sigma})$, by using \eqref{3-4.3}, we have
\begin{align}
\label{3-4.4}
&\|\psi^{\alpha}(\omega_{\ell^\alpha}, \xi_{1})-\psi^{\alpha}(\omega_{\ell^\alpha},  \xi_{2})\|_{\sigma}\nonumber\\ &\leq L\int_{-\infty}^0 \Big[ \Big( \frac{\sigma}{-s} \Big)^{\sigma} + \lambda_{N+1}^{\sigma} \Big] e^{(\lambda_{N+1}-\beta)s}ds \|\bar{u}^{\alpha}(t, \omega_{\ell^\alpha}, \xi_{1} )-\bar{u}^{\alpha}(t, \omega_{\ell^\alpha}, \xi_{2})\|_{\mathcal{C}_{\beta}^{\alpha,-}}\nonumber\\
&\leq L(1-\mu)^{-1}\left( \frac{\sigma^{\sigma}\Gamma(1-\sigma)}{(\lambda_{N+1}-\beta)^{1-\sigma}}+\frac{\lambda_{N+1}^{\sigma}}{\lambda_{N+1}-\beta}\right)\|\xi_{1}- \xi_{2}\|_{\sigma}\nonumber\\
&\leq \frac{\mu}{2(1-\mu)}\|\xi_{1}- \xi_{2}\|_{\sigma},
\end{align}
where the last inequality follows from \eqref{3-4.2}. Hence, $\psi^{\alpha}(\omega_{\ell^\alpha}, \xi)$ is Lipschitz continuous in $\xi\in PD(A^{\sigma})$. \par
 Then, we show that $ \mathscr{M}^{\alpha}(\omega_{\ell^\alpha})$ is a random set. Since $D(A^{\sigma})$ is a dense subspace of separable Banach space $H$, $D(A^{\sigma})$ is also separable.
Let $D(A^{\sigma})_c$ be a countable dense set of $D(A^{\sigma})$. For each $x \in D(A^{\sigma})$, we observe that
\begin{align*}
\omega_{\ell^\alpha} \mapsto\inf_{y\in D(A^{\sigma})}\big\lbrace \|x - (Py + \psi^{\alpha}(\omega_{\ell^\alpha}, Py))\|_{\sigma}\big\rbrace=\inf_{y\in D(A^{\sigma})_c}\big\lbrace \|x-(Py+\psi^{\alpha}(\omega_{\ell^\alpha}, Py))\|_{\sigma}\big\rbrace.
\end{align*}
By \cite[Theorem III.14]{Castaing77}, we obtain that $\mathscr{M}^{\alpha}(\omega_{\ell^\alpha})$ is $\mathcal{F}$-measurable. Finally, we prove that $ \mathscr{M}^{\alpha}(\omega_{\ell^\alpha})$ is invariant, i.e., for any $s \geq 0$,
\begin{align*}
u^{\alpha}(s, \omega_{\ell^\alpha}, \mathscr{M}^{\alpha}(\omega_{\ell^\alpha})) \subset \mathscr{M}^{\alpha}(\theta_s \omega_{\ell^\alpha}).
\end{align*}
For any $x \in  \mathscr{M}^{\alpha}(\omega_{\ell^\alpha})$, we have $u^{\alpha}(t, \theta_s \omega_{\ell^\alpha}, u^{\alpha}(s, \omega_{\ell^\alpha}, x)) = u^{\alpha}(t + s, \omega_{\ell^\alpha}, x)$. Hence
$u^{\alpha}(s, \omega_{\ell^\alpha}, x) \in \mathscr{M}^{\alpha}(\theta_s \omega_{\ell^\alpha})$, which implies that $u^{\alpha}(s, \omega_{\ell^\alpha}, \mathscr{M}^{\alpha}(\omega_{\ell^\alpha})) \subset \mathscr{M}^{\alpha}(\theta_s \omega_{\ell^\alpha})$. Thus $\mathscr{M}^{\alpha}(\omega_{\ell^\alpha})$ is a Lipschitz invariant manifold.\\
$\mathbf{Step\ 3}$. We prove that $\mathscr{M}^{\alpha}(\omega_{\ell^\alpha})$ has the asymptotic completeness property.
Let $\mathcal{C}_{\beta}^{\alpha,+}$ is a Banach space given by
\begin{align*}
\mathcal{C}_{\beta}^{\alpha,+}= \Big\{ \varphi \in C([0,+\infty), D(A^{\sigma}) \ \big| \sup_{t \geq 0} \big\{e^{\beta t - \int_0^t z(\theta_r \omega_{\ell^\alpha}) dr} \|\varphi(t)\|_{\sigma} \big\} < +\infty \Big\}
\end{align*}
with the norm
\begin{align*}
\|\varphi\|_{\mathcal{C}_{\beta}^{\alpha,+}} = \sup_{t \geq 0} \big\{e^{\beta t - \int_0^t z(\theta_r \omega_{\ell^\alpha}) dr} \|\varphi(t)\|_{\sigma} \big\}.
\end{align*}\par
For any $y^{\alpha} \in \mathcal{C}_{\beta}^{\alpha,+}$, $x \in D(A^{\sigma})$, using the same procedure as Lemma \ref{lem:3-4.1}, we observe that $y^{\alpha} + u^{\alpha}(s, \omega_{\ell^\alpha}, x)$
is a solution of equation \eqref{eq:3-2.4} if and only if there exists some $p \in Q D(A^{\sigma})$ such that $y^{\alpha}$ satisfies
\begin{align}
\label{3-4.5}
y^{\alpha}(t) = &e^{-At + \int_0^t z(\theta_r \omega_{\ell^\alpha}) dr} p + \int_0^t e^{-A(t-s) + \int_s^t z(\theta_r \omega_{\ell^\alpha}) dr} Q \Delta G(\theta_s \omega_{\ell^\alpha}, y^{\alpha}(s),u^\alpha(s, \omega_{\ell^\alpha}, x))ds \nonumber \\
& + \int_{+\infty}^t e^{-A(t-s)+ \int_s^t z(\theta_r \omega_{\ell^\alpha}) dr} P \Delta G(\theta_s \omega_{\ell^\alpha}, y^{\alpha}(s),u^\alpha(s, \omega_{\ell^\alpha}, x))ds,
\end{align}
where
\begin{align*}
\Delta G(\theta_s \omega_{\ell^\alpha}, y^{\alpha}(s),u^\alpha(s, \omega_{\ell^\alpha}, x))= G(\theta_s \omega_{\ell^\alpha}, y^{\alpha}(s)+u^\alpha(s, \omega_{\ell^\alpha}, x))- G(\theta_s \omega_{\ell^\alpha}, u^\alpha(s, \omega_{\ell^\alpha}, x)).
\end{align*}
Denote the right hand side of \eqref{3-4.5} by $\mathcal{P}(y^{\alpha}, \omega_{\ell^\alpha}, p)$, by \eqref{3-2.1} and \eqref{3-2.2} we get for $t\geq0$,
\begin{align*}
&e^{\beta t - \int_0^t z(\theta_r \omega_{\ell^\alpha}) dr}\Big\|\mathcal{P}(y^{\alpha}, \omega_{\ell^\alpha}, p)\Big\|_{\sigma}\\
&\leq e^{(\beta-\lambda_{N+1}) t}\|p\|_{\sigma}+L \int_0^t \Big[ \Big( \frac{\sigma}{t-s} \Big)^{\sigma} + \lambda_{N+1}^{\sigma} \Big] e^{(\beta-\lambda_{N+1})(t-s)}\|y^{\alpha}\|_{\mathcal{C}_{\beta}^{\alpha,+}}ds
\\&\ \ \ +L\lambda_{N}^{\sigma}\int_t^{+\infty}e^{(\beta-\lambda_{N})(t-s)}\|y^{\alpha}\|_{\mathcal{C}_{\beta}^{\alpha,+}}ds \\
&\leq \|p\|_{\sigma}+L\left( \frac{\lambda_{N}^{\sigma}}{\beta-\lambda_{N}}+\frac{\sigma^{\sigma}\Gamma(1-\sigma)}{(\lambda_{N+1}-\beta)^{1-\sigma}}+\frac{\lambda_{N+1}^{\sigma}}{\lambda_{N+1}-\beta}\right)\|y^{\alpha}\|_{\mathcal{C}_{\beta}^{\alpha,+}}\\
&\leq \|p\|_{\sigma}+\mu\|y^{\alpha}\|_{\mathcal{C}_{\beta}^{\alpha,+}}.
\end{align*}
This shows that $\mathcal{P}(y^{\alpha}, \omega_{\ell^\alpha}, p)$ maps $\mathcal{C}_{\beta}^{\alpha,+}$ into itself. Then, for any $y_{1}^{\alpha},y_{2}^{\alpha}\in \mathcal{C}_{\beta}^{\alpha,+} $, we have
\begin{align*}
\Big\|\mathcal{P}(y_{1}^{\alpha}, \omega_{\ell^\alpha}, p)-\mathcal{P}(y_{2}^{\alpha}, \omega_{\ell^\alpha}, p)\Big\|_{\mathcal{C}_{\beta}^{\alpha,+}} \leq \mu\|y_{1}^{\alpha}-y_{2}^{\alpha}\|_{\mathcal{C}_{\beta}^{\alpha,+}},
\end{align*}
which implies that $\mathcal{P}(y^{\alpha}, \omega_{\ell^\alpha}, p)$ is a contraction map. By the contraction mapping principle, $\mathcal{P}(y^{\alpha}, \omega_{\ell^\alpha}, p)$ has a unique fixed point $\bar{y}^{\alpha}(\cdot, \omega_{\ell^\alpha}, p)$ and satisfies
\begin{align}
\label{3-4.6}
\|\bar{y}^{\alpha}(\cdot, \omega_{\ell^\alpha}, p)\|_{\mathcal{C}_{\beta}^{\alpha,+}} \leq (1-\mu)^{-1}\|p\|_{\sigma}.
\end{align}
Consequently, we find that $\bar{y}^{\alpha}(\cdot, \omega_{\ell^\alpha}, p) + u^{\alpha}(\cdot, \omega_{\ell^\alpha}, x)$ is a solution of equation \eqref{eq:3-2.4}. This implies that there exists a $\tilde{x} \in D(A^{\sigma})$ such that
\begin{align}
\label{3-4.7}
u^{\alpha}(t, \omega_{\ell^\alpha}, \tilde{x})=\bar{y}^{\alpha}(t, \omega_{\ell^\alpha}, p) + u^{\alpha}(t, \omega_{\ell^\alpha}, x),
\end{align}
where $\tilde{x}=\bar{y}^{\alpha}(0, \omega_{\ell^\alpha}, p)+x$.\par
Following from Step 2, we can see that $\tilde{x} \in \mathscr{M}^{\alpha}(\omega_{\ell^\alpha})$ if and only if $Q\tilde{x}=\psi^{\alpha}(\omega_{\ell^\alpha}, P\tilde{x})$. Hence, letting
\begin{align*}
p = \psi^{\alpha}(\omega_{\ell^\alpha}, \int_{+\infty}^0 e^{As+ \int_s^0 z(\theta_r \omega_{\ell^\alpha}) dr} P \Delta G(\theta_s \omega_{\ell^\alpha}, y^{\alpha}(s),u^\alpha(s, \omega_{\ell^\alpha}, x))ds+Px)-Qx,
\end{align*}
we have that $\tilde{x} \in \mathscr{M}^{\alpha}(\omega_{\ell^\alpha})$. Combining \eqref{3-4.6} and \eqref{3-4.7}, we obtain that
\begin{align*}
\|u^{\alpha}(t, \omega_{\ell^\alpha}, x)-u^{\alpha}(t, \omega_{\ell^\alpha}, \tilde{x})\|_{\sigma} \leq e^{-\beta t + \int_0^t z(\theta_r \omega_{\ell^\alpha}) dr}(1-\mu)^{-1}\|x-\tilde{x}\|_{\sigma}.
\end{align*}
According to Lemma \ref{lem:3-2.1}, we have that $e^{ \int_0^t z(\theta_r \omega_{\ell^\alpha}) dr}$ satisfies the sublinear growth, there exists a random variable $C(\omega_{\ell^\alpha})$ such that
\begin{align*}
e^{ \int_0^t z(\theta_r \omega_{\ell^\alpha}) dr}\leq C(\omega_{\ell^\alpha})e^{\frac{1}{2}\beta t}.
\end{align*}
Thus we arrive at the asymptotic completeness property of $\mathscr{M}^{\alpha}(\omega_{\ell^\alpha})$ as
\begin{align}
\label{3-4.8}
\|u^{\alpha}(t, \omega_{\ell^\alpha}, x)-u^{\alpha}(t, \omega_{\ell^\alpha}, \tilde{x})\|_{\sigma} \leq  C(\omega_{\ell^\alpha})(1-\mu)^{-1}e^{-\frac{1}{2}\beta t}\|x-\tilde{x}\|_{\sigma}.
\end{align}
Therefore, $\mathscr{M}^{\alpha}(\omega_{\ell^\alpha})$ is a Lipschitz inertial manifolds for equation \eqref{eq:3-2.4}.\\
$\mathbf{Step\ 4}$. We show the existence of Lipschitz inertial manifolds for equation \eqref{eq:3-1.2}. Notice that $\widetilde{\mathscr{M}}^{\alpha}(\omega_{\ell^\alpha})=T^{-1}(\omega_{\ell^\alpha}, \mathscr{M}^{\alpha}(\omega_{\ell^\alpha}))$ is a Lipschitz invariant manifold of equation \eqref{eq:3-1.2}, and
\begin{align*}
\widetilde{\mathscr{M}}^{\alpha}(\omega_{\ell^\alpha})
&=\Big\lbrace T^{-1}(\omega_{\ell^\alpha},\xi+\psi^{\alpha}(\omega_{\ell^\alpha}, \xi))|\xi\in PD(A^{\sigma})\Big\rbrace\\
&=\left\lbrace e^{z(\omega_{\ell^\alpha})}(\xi+\psi^{\alpha}(\omega_{\ell^\alpha}, \xi))|\xi\in PD(A^{\sigma})\right\rbrace\\
&=\left\lbrace \xi+e^{z(\omega_{\ell^\alpha})}\psi^{\alpha}(\omega_{\ell^\alpha}, e^{-z(\omega_{\ell^\alpha})}\xi)|\xi\in PD(A^{\sigma})\right\rbrace.
\end{align*}
In fact, we find that
\begin{align*}
v^{\alpha}(t, \omega_{\ell^\alpha},\widetilde{\mathscr{M}}^{\alpha}(\omega_{\ell^\alpha}))
&= T^{-1}(\theta_t \omega_{\ell^\alpha}, u^{\alpha}(t, \omega_{\ell^\alpha}, T(\omega_{\ell^\alpha},\widetilde{\mathscr{M}}^{\alpha}(\omega_{\ell^\alpha})))) \\
&= T^{-1}(\theta_t \omega_{\ell^\alpha}, u^{\alpha}(t, \omega_{\ell^\alpha}, \mathscr{M}^{\alpha}(\omega_{\ell^\alpha}))) \\ &\subset T^{-1}(\theta_t \omega_{\ell^\alpha}, \mathscr{M}^{\alpha}(\theta_t \omega_{\ell^\alpha}))
= \widetilde{\mathscr{M}}^{\alpha}(\theta_t \omega_{\ell^\alpha}).
\end{align*}\par
Finally, we prove the asymptotic completeness property of $\widetilde{\mathscr{M}}^{\alpha}(\omega_{\ell^\alpha})$. Given $x \in D(A^{\sigma})$, we observe that
\begin{align*}
v^{\alpha}(t, \omega_{\ell^\alpha},x)
&= T^{-1}(\theta_t \omega_{\ell^\alpha}, u^{\alpha}(t, \omega_{\ell^\alpha}, T(\omega_{\ell^\alpha},x))) \\
&=e^{z(\omega_{\ell^\alpha})}u^{\alpha}(t, \omega_{\ell^\alpha}, e^{-z(\omega_{\ell^\alpha})}x).
\end{align*}
Let $y=e^{-z(\omega_{\ell^\alpha})}x$, by using \eqref{3-4.8}, there exists a $\tilde{y}\in \mathscr{M}^{\alpha}(\omega_{\ell^\alpha})$ such that
\begin{align}
\label{3-4.9}
\|u^{\alpha}(t, \omega_{\ell^\alpha}, y)-u^{\alpha}(t, \omega_{\ell^\alpha}, \tilde{y})\|_{\sigma} \leq  C(\omega_{\ell^\alpha})(1-\mu)^{-1}e^{-\frac{1}{2}\beta t}\|y-\tilde{y}\|_{\sigma}.
\end{align}
Choose $\tilde{x}=e^{z(\omega_{\ell^\alpha})}\tilde{y}$, for $\tilde{x}\in \widetilde{\mathscr{M}}^{\alpha}(\omega_{\ell^\alpha})$ and along with \eqref{3-4.9} yields
\begin{align*}
\|v^{\alpha}(t, \omega_{\ell^\alpha}, x)-v^{\alpha}(t, \omega_{\ell^\alpha}, \tilde{x})\|_{\sigma} \leq  C(\omega_{\ell^\alpha})(1-\mu)^{-1}e^{-\frac{1}{2}\beta t}\|x-\tilde{x}\|_{\sigma}.
\end{align*}
Hence $\widetilde{\mathscr{M}}^{\alpha}(\omega_{\ell^\alpha})$ is the Lipschitz inertial manifold of equation \eqref{eq:3-1.2}.\\
$\mathbf{Step\ 5}$. If $F\in C^{1}$, we prove that $\psi^{\alpha}(\omega_{\ell^\alpha}, \xi)$ is $C^{1}$ in $\xi$.\par
We first choose $\delta>0$ such that $\lambda_{N}<\beta-2\delta<\beta-\delta<\lambda_{N+1}$ and
\begin{align*}
L\left( \frac{\lambda_{N}^{\sigma}}{\beta-j\delta-\lambda_{N}}+\frac{\sigma^{\sigma}\Gamma(1-\sigma)}{(\lambda_{N+1}-\beta+j\delta)^{1-\sigma}}+\frac{\lambda_{N+1}^{\sigma}}{\lambda_{N+1}-\beta+j\delta}\right)
\leq\mu<1, j=1,2.
\end{align*}
This implies that $\mathcal{Y}(\cdot, \omega_{\ell^\alpha}, \xi)$ has a unique fixed point $\bar{u}^{\alpha}(\cdot,\omega_{\ell^\alpha},\xi)\in \mathcal{C}_{\beta-2\delta}^{\alpha,-}\subset\mathcal{C}_{\beta-\delta}^{\alpha,-}\subset\mathcal{C}_{\beta}^{\alpha,-}$. Next, we prove that $u^{\alpha}(\cdot,\omega_{\ell^\alpha},\xi)$ is $C^{1}$ from $PD(A^{\sigma})$ to $\mathcal{C}_{\beta}^{\alpha,-}$. \par
For any $u\in \mathcal{C}_{\beta-\delta}^{\alpha,-}$,  we define
\begin{align*}
\mathcal{G}u(t)=&\int_0^te^{-A(t-s)+\int_{s}^{t}z(\theta_{r}\omega_{\ell^\alpha})dr}PD_{u}G(\theta_{s}\omega_{\ell^\alpha}, \bar{u}^{\alpha}(s, \omega_{\ell^\alpha}, \xi_{0}))u(s)ds\\&+\int_{-\infty}^te^{-A(t-s)+\int_{s}^{t}z(\theta_{r}\omega_{\ell^\alpha})dr}QD_{u}G(\theta_{s}\omega_{\ell^\alpha}, \bar{u}^{\alpha}(s, \omega_{\ell^\alpha}, \xi_{0}))u(s)ds.
\end{align*}
We find that
\begin{align*}
e^{(\beta-\delta) t - \int_0^t z(\theta_r \omega_{\ell^\alpha}) dr}\|\mathcal{G}u(t)\|_{\sigma}\leq&L\lambda_{N}^{\sigma}\int_t^0e^{(\beta-\delta-\lambda_{N})(t-s)} \|u\|_{\mathcal{C}_{\beta-\delta}^{\alpha,-}}ds
\\&+ L\int_{-\infty}^t \Big[ \Big( \frac{\sigma}{t-s} \Big)^{\sigma} + \lambda_{N+1}^{\sigma} \Big] e^{(\beta-\delta-\lambda_{N+1})(t-s)} \|u\|_{\mathcal{C}_{\beta-\delta}^{\alpha,-}}ds.
\end{align*}
Then, $ \mathcal{G}$ is a bounded linear operator from $\mathcal{C}_{\beta-\delta}^{\alpha,-}$ to itself with the norm
\begin{align*}
\|\mathcal{G}\|\leq L\left( \frac{\lambda_{N}^{\sigma}}{\beta-\delta-\lambda_{N}}+\frac{\sigma^{\sigma}\Gamma(1-\sigma)}{(\lambda_{N+1}-\beta+\delta)^{1-\sigma}}+\frac{\lambda_{N+1}^{\sigma}}{\lambda_{N+1}-\beta+\delta}\right)
\leq\mu<1,
\end{align*}
which implies that $id-\mathcal{G}$ has a bounded inverse in $\mathcal{C}_{\beta-\delta}^{\alpha,-}$. Let
\begin{align*}
\mathcal{H}=&\int_0^te^{-A(t-s)+\int_{s}^{t}z(\theta_{r}\omega_{\ell^\alpha})dr}P\Big[G(\theta_{s}\omega_{\ell^\alpha}, \bar{u}^{\alpha}(s, \omega_{\ell^\alpha}, \xi))-G(\theta_{s}\omega_{\ell^\alpha}, \bar{u}^{\alpha}(s, \omega_{\ell^\alpha}, \xi_{0}))\\&\ \ \ \ -D_{u}G(\theta_{s}\omega_{\ell^\alpha}, \bar{u}^{\alpha}(s, \omega_{\ell^\alpha}, \tilde{\xi}))\big(\bar{u}^{\alpha}(s, \omega_{\ell^\alpha}, \xi)-\bar{u}^{\alpha}(s, \omega_{\ell^\alpha}, \xi_{0})\big)\Big]ds\\&+\int_{-\infty}^te^{-A(t-s)+\int_{s}^{t}z(\theta_{r}\omega_{\ell^\alpha})dr}Q\Big[G(\theta_{s}\omega_{\ell^\alpha}, \bar{u}^{\alpha}(s, \omega_{\ell^\alpha}, \xi))-G(\theta_{s}\omega_{\ell^\alpha}, \bar{u}^{\alpha}(s, \omega_{\ell^\alpha}, \xi_{0}))\\&\ \ \ \ -D_{u}G(\theta_{s}\omega_{\ell^\alpha}, \bar{u}^{\alpha}(s, \omega_{\ell^\alpha}, \xi_{0}))\big(\bar{u}^{\alpha}(s, \omega_{\ell^\alpha}, \xi)-\bar{u}^{\alpha}(s, \omega_{\ell^\alpha}, \xi_{0})\big)\Big]ds.
\end{align*}
Denote $\mathcal{S}=e^{At+\int_{0}^{t}z(\theta_{r}\omega_{\ell^\alpha})dr}$, which is a bounded operator from $PD(A^{\sigma})$ to $\mathcal{C}_{\beta-\delta}^{\alpha,-}$. We claim that
 \begin{align*}
 \|\mathcal{H}\|_{\mathcal{C}_{\beta-\delta}^{\alpha,-}}=o(\|\xi-\xi_{0}\|_{\sigma})\ as \ \xi\rightarrow\xi_{0}.
 \end{align*}
Then we have
\begin{align*}
&\bar{u}^{\alpha}(\cdot,\omega_{\ell^\alpha},\xi)-\bar{u}^{\alpha}(\cdot,\omega_{\ell^\alpha},\xi_{0})-\mathcal{G}(\bar{u}^{\alpha}(\cdot,\omega_{\ell^\alpha},\xi)-\bar{u}^{\alpha}(\cdot,\omega_{\ell^\alpha},\xi_{0})\big)
\\&=\mathcal{S}(\xi-\xi_{0})+o(\|\xi-\xi_{0}\|),\ as\ \xi\rightarrow\xi_{0},
\end{align*}
which implies that
\begin{align*}
\bar{u}^{\alpha}(\cdot,\omega_{\ell^\alpha},\xi)-\bar{u}^{\alpha}(\cdot,\omega_{\ell^\alpha},\xi_{0})=(id-\mathcal{G})^{-1}\mathcal{S}(\xi-\xi_{0})+o(\|\xi-\xi_{0}\|).
\end{align*}
Now we prove that $\|\mathcal{H}\|_{\mathcal{C}_{\beta-\delta}^{\alpha,-}}=o(\|\xi-\xi_{0}\|_{\sigma})\ as \ \xi\rightarrow\xi_{0}$. We decompose $e^{(\beta-\delta) t - \int_0^t z(\theta_r \omega_{\ell^\alpha}) dr}\mathcal{H}$ into a sum of four terms, i.e.,
\begin{align*}
 e^{(\beta-\delta) t - \int_0^t z(\theta_r \omega_{\ell^\alpha}) dr}\mathcal{H}=\mathcal{H}_{1}+\mathcal{H}_{2}+\mathcal{H}_{3}+\mathcal{H}_{4},\ \forall t\leq0,
\end{align*}
where
\begin{align*}
\mathcal{H}_{1}=&-e^{(\beta-\delta)t-\int_{0}^{t}z(\theta_{r}\omega_{\ell^\alpha})dr}\int_t^{-N_{1}}e^{-A(t-s)+\int_{s}^{t}z(\theta_{r}\omega_{\ell^\alpha})dr}P\Big[G(\theta_{s}\omega_{\ell^\alpha}, \bar{u}^{\alpha}(s, \omega_{\ell^\alpha}, \xi))\\&\ \ \ \ -G(\theta_{s}\omega_{\ell^\alpha}, \bar{u}^{\alpha}(s, \omega_{\ell^\alpha}, \xi_{0}))-D_{u}G(\theta_{s}\omega_{\ell^\alpha}, \bar{u}^{\alpha}(s, \omega_{\ell^\alpha}, \xi_{0}))\big(\bar{u}^{\alpha}(s, \omega_{\ell^\alpha}, \xi)-\bar{u}^{\alpha}(s, \omega_{\ell^\alpha}, \xi_{0})\big)\Big]ds
\end{align*}
for $t< -N_{1}$ and $\mathcal{H}_{1}=0$ for  $t\geq -N_{1}$,
\begin{align*}
\mathcal{H}_{2}=&-e^{(\beta-\delta)t-\int_{0}^{t}z(\theta_{r}\omega_{\ell^\alpha})dr}\int_{-N_{1}}^0e^{-A(t-s)+\int_{s}^{t}z(\theta_{r}\omega_{\ell^\alpha})dr}P\Big[G(\theta_{s}\omega_{\ell^\alpha}, \bar{u}^{\alpha}(s, \omega_{\ell^\alpha}, \xi))\\&\ \ \ \ -G(\theta_{s}\omega_{\ell^\alpha}, \bar{u}^{\alpha}(s, \omega_{\ell^\alpha}, \xi_{0}))-D_{u}G(\theta_{s}\omega_{\ell^\alpha}, \bar{u}^{\alpha}(s, \omega_{\ell^\alpha}, \xi_{0}))\big(\bar{u}^{\alpha}(s, \omega_{\ell^\alpha}, \xi)-\bar{u}^{\alpha}(s, \omega_{\ell^\alpha}, \xi_{0})\big)\Big]ds
\end{align*}
for $t< -N_{1}$ and change $-N_{1}$ to $t$ for $t\geq -N_{1}$,
\begin{align*}
\mathcal{H}_{3}=&e^{(\beta-\delta)t-\int_{0}^{t}z(\theta_{r}\omega_{\ell^\alpha})dr}\int_{-\infty}^{-N_{2}}e^{-A(t-s)+\int_{s}^{t}z(\theta_{r}\omega_{\ell^\alpha})dr}Q\Big[G(\theta_{s}\omega_{\ell^\alpha}, \bar{u}^{\alpha}(s, \omega_{\ell^\alpha}, \xi))\\&\ \ \ \ -G(\theta_{s}\omega_{\ell^\alpha}, \bar{u}^{\alpha}(s, \omega_{\ell^\alpha}, \xi_{0}))-D_{u}G(\theta_{s}\omega_{\ell^\alpha}, \bar{u}^{\alpha}(s, \omega_{\ell^\alpha}, \xi_{0}))\big(\bar{u}^{\alpha}(s, \omega_{\ell^\alpha}, \xi)-\bar{u}^{\alpha}(s, \omega_{\ell^\alpha}, \xi_{0})\big)\Big]ds
\end{align*}
for $t> -N_{2}$ and change $-N_{2}$ to $t$ for $t\leq -N_{2}$,
\begin{align*}
\mathcal{H}_{4}=&e^{(\beta-\delta)t-\int_{0}^{t}z(\theta_{r}\omega_{\ell^\alpha})dr}\int_{-N_{2}}^te^{-A(t-s)+\int_{s}^{t}z(\theta_{r}\omega_{\ell^\alpha})dr}Q\Big[G(\theta_{s}\omega_{\ell^\alpha}, \bar{u}^{\alpha}(s, \omega_{\ell^\alpha}, \xi))\\&\ \ \ \ -G(\theta_{s}\omega_{\ell^\alpha}, \bar{u}^{\alpha}(s, \omega_{\ell^\alpha}, \xi_{0}))-D_{u}G(\theta_{s}\omega_{\ell^\alpha}, \bar{u}^{\alpha}(s, \omega_{\ell^\alpha}, \xi_{0}))\big(\bar{u}^{\alpha}(s, \omega_{\ell^\alpha}, \xi)-\bar{u}^{\alpha}(s, \omega_{\ell^\alpha}, \xi_{0})\big)\Big]ds
\end{align*}
for $t>-N_{2}$ and $\mathcal{H}_{4}=0$ for  $t\leq -N_{2}$. Here $N_{1}$ and $N_{2}$ are sufficiently large positive number to be chosen later. We find that for $t< -N_{1}$,
\begin{align*}
\|\mathcal{H}_{1}\|_{\sigma}&\leq L\lambda_{N}^{\sigma}e^{-\delta N_{1}}\int_t^{-N_{1}}e^{(\beta-2\delta-\lambda_{N})(t-s)}ds \|\bar{u}^{\alpha}(\cdot, \omega_{\ell^\alpha}, \xi)-\bar{u}^{\alpha}(\cdot, \omega_{\ell^\alpha}, \xi_{0})\|_{\mathcal{C}_{\beta-2\delta}^{\alpha,-}}
\\&\leq \frac{2L\lambda_{N}^{\sigma}}{\beta-2\delta-\lambda_{N}}(1-\mu)^{-1}e^{-\delta N_{1}}\|\xi-\xi_{0}\|_{\sigma}.
\end{align*}
For any $\varepsilon>0$, choose $N_{1}$ so large that
\begin{align}
\label{3-4.10}
\sup_{t\leq0}\|\mathcal{H}_{1}\|_{\sigma}\leq\frac{\varepsilon}{4}\|\xi-\xi_{0}\|_{\sigma}.
\end{align}
Fixing such $N_{1}$, we get that
\begin{align*}
\|\mathcal{H}_{2}\|_{\sigma}&\leq \lambda_{N}^{\sigma}\int_{-N_{1}}^0e^{(\beta-\delta-\lambda_{N})(t-s)} \int_{0}^{1}\|D_{u}G(\theta_{s}\omega_{\ell^\alpha}, \tau \bar{u}^{\alpha}(s, \omega_{\ell^\alpha}, \xi)+(1-\tau) \bar{u}^{\alpha}(s, \omega_{\ell^\alpha}, \xi_{0})) \\&\ \ \ \ \ \ \ \ \ -D_{u}G(\theta_{s}\omega_{\ell^\alpha}, \bar{u}^{\alpha}(s, \omega_{\ell^\alpha}, \xi_{0}))\|d\tau ds \|\bar{u}^{\alpha}(\cdot, \omega_{\ell^\alpha}, \xi)-\bar{u}^{\alpha}(\cdot, \omega_{\ell^\alpha}, \xi_{0})\|_{\mathcal{C}_{\beta-2\delta}^{\alpha,-}}
\\&\leq \lambda_{N}^{\sigma}(1-\mu)^{-1}\|\xi-\xi_{0}\|_{\sigma}\int_{-N_{1}}^0e^{(\beta-\delta-\lambda_{N})(t-s)}\\&\ \ \ \ \ \times \int_{0}^{1}\|D_{u}G(\theta_{s}\omega_{\ell^\alpha}, \tau \bar{u}^{\alpha}(s, \omega_{\ell^\alpha}, \xi)+(1-\tau) \bar{u}^{\alpha}(s, \omega_{\ell^\alpha}, \xi_{0})) \\&\ \ \ \ \ \ \ \ \ -D_{u}G(\theta_{s}\omega_{\ell^\alpha}, \bar{u}^{\alpha}(s, \omega_{\ell^\alpha}, \xi_{0}))\|d\tau ds.
\end{align*}
Since the continuity of the integrand, for $\varepsilon>0$, there exists $\rho_{1}>0$ such that if $\|\xi-\xi_{0}\|_{\sigma}\leq\rho_{1}$,
\begin{align}
\label{3-4.11}
\sup_{t\leq0}\|\mathcal{H}_{2}\|_{\sigma}\leq\frac{\varepsilon}{4}\|\xi-\xi_{0}\|_{\sigma}.
\end{align}
Analogously, by choosing $N_{2}$ to be sufficiently large, we have
\begin{align}
\label{3-4.12}
\sup_{t\leq0}\|\mathcal{H}_{3}\|_{\sigma}\leq\frac{\varepsilon}{4}\|\xi-\xi_{0}\|_{\sigma}.
\end{align}
Fix $N_{2}$, there exists $\rho_{2}>0$ such that if $\|\xi-\xi_{0}\|_{\sigma}\leq\rho_{2}$,
\begin{align}
\label{3-4.13}
\sup_{t\leq0}\|\mathcal{H}_{4}\|_{\sigma}\leq\frac{\varepsilon}{4}\|\xi-\xi_{0}\|_{\sigma}.
\end{align}
Taking $\rho=\min\left\{\rho_{1}, \rho_{2}\right\}$ and combining \eqref{3-4.10}-\eqref{3-4.13}, we obtain that
\begin{align*}
\|\mathcal{H}\|_{\mathcal{C}_{\beta-\delta}^{\alpha,-}}=o( \|\xi-\xi_{0}\|_{\sigma})\ as \ \xi\rightarrow\xi_{0} .
\end{align*}
Hence, $\bar{u}^{\alpha}(\cdot,\omega_{\ell^\alpha},\xi)$ is differentiable in $\xi$ and $D_{\xi}\bar{u}^{\alpha}(\cdot,\omega_{\ell^\alpha},\xi)\in\mathcal{L}(PD(A^{\sigma}), \mathcal{C}_{\beta-\delta}^{\alpha,-})$.\par
From \eqref{3-4.1}, we have
\begin{align*}
D_{\xi}\bar{u}^{\alpha}(t,\omega_{\ell^\alpha},\xi)=&e^{-At+\int_{0}^{t}z(\theta_{r}\omega_{\ell^\alpha})dr}P\\
&+\int_0^te^{-A(t-s)+\int_{s}^{t}z(\theta_{r}\omega_{\ell^\alpha})dr}PD_{u}G(\theta_{s}\omega_{\ell^\alpha}, \bar{u}^{\alpha}(s, \omega_{\ell^\alpha}, \xi))D_{\xi}\bar{u}^{\alpha}(s,\omega_{\ell^\alpha},\xi)ds\\
&+\int_{-\infty}^te^{-A(t-s)+\int_{s}^{t}z(\theta_{r}\omega_{\ell^\alpha})dr}QD_{u}G(\theta_{s}\omega_{\ell^\alpha}, \bar{u}^{\alpha}(s, \omega_{\ell^\alpha}, \xi))D_{\xi}\bar{u}^{\alpha}(s,\omega_{\ell^\alpha},\xi)ds.
\end{align*}
For $\xi\in PD(A^{\sigma})$, we define the operator $\widetilde{\mathcal{G}}:\mathcal{L}(PD(A^{\sigma}), \ \mathcal{C}_{\beta}^{\alpha,-})\rightarrow \mathcal{L}(PD(A^{\sigma}), \ \mathcal{C}_{\beta}^{\alpha,-})$
\begin{align*}
\widetilde{\mathcal{G}}u(t)=&\int_0^te^{-A(t-s)+\int_{s}^{t}z(\theta_{r}\omega_{\ell^\alpha})dr}PD_{u}G(\theta_{s}\omega_{\ell^\alpha}, \bar{u}^{\alpha}(s, \omega_{\ell^\alpha}, \xi))u(s)ds\\&+\int_{-\infty}^te^{-A(t-s)+\int_{s}^{t}z(\theta_{r}\omega_{\ell^\alpha})dr}QD_{u}G(\theta_{s}\omega_{\ell^\alpha}, \bar{u}^{\alpha}(s, \omega_{\ell^\alpha}, \xi))u(s)ds.
\end{align*}
Similarly, we have that $\|\widetilde{\mathcal{G}}\|<1$. This yields that $id-\widetilde{\mathcal{G}}$ has a bounded inverse in $\mathcal{L}(PD(A^{\sigma}), \mathcal{C}_{\beta}^{\alpha,-})$. \par
We find that
\begin{align}
\label{3-4.14}
&D_{\xi}\bar{u}^{\alpha}(t,\omega_{\ell^\alpha},\xi)-D_{\xi}\bar{u}^{\alpha}(t,\omega_{\ell^\alpha},\xi_{0})\nonumber\\
&=\widetilde{\mathcal{G}}(D_{\xi}\bar{u}^{\alpha}(t,\omega_{\ell^\alpha},\xi)-D_{\xi}\bar{u}^{\alpha}(t,\omega_{\ell^\alpha},\xi_{0}))+\widetilde{\mathcal{H}},
\end{align}
where
\begin{align*}
\widetilde{\mathcal{H}}=&\int_0^te^{-A(t-s)+\int_{s}^{t}z(\theta_{r}\omega_{\ell^\alpha})dr}P\big[D_{u}G(\theta_{s}\omega_{\ell^\alpha}, \bar{u}^{\alpha}(s, \omega_{\ell^\alpha}, \xi))\\&\ \ \  -D_{u}G(\theta_{s}\omega_{\ell^\alpha},\bar{u}^{\alpha}(s,\omega,\xi_{0}))\big]D_{\xi}\bar{u}^{\alpha}(s,\omega_{\ell^\alpha},\xi_{0})ds\\
&+\int_{-\infty}^te^{-A(t-s)+\int_{s}^{t}z(\theta_{r}\omega_{\ell^\alpha})dr}Q\big[D_{u}G(\theta_{s}\omega_{\ell^\alpha}, \bar{u}^{\alpha}(s, \omega_{\ell^\alpha}, \xi))\\&\ \ \  -D_{u}G(\theta_{s}\omega_{\ell^\alpha},\bar{u}^{\alpha}(s,\omega,\xi_{0}))\big]D_{\xi}\bar{u}^{\alpha}(s,\omega_{\ell^\alpha},\xi_{0})ds.
\end{align*}
Using the same procedure as for $\mathcal{H}$, we get that $\|\widetilde{\mathcal{H}}\|_{\mathcal{L}(PD(A^{\sigma}),  \mathcal{C}_{\beta}^{\alpha,-})}=o(1)\ as \ \xi\rightarrow\xi_{0}$. In view of \eqref{3-4.10}, $D_{\xi}\bar{u}^{\alpha}(\cdot,\omega_{\ell^\alpha},\xi)$ is continuous with respect to $\xi$. Hence  $\bar{u}^{\alpha}(\cdot,\omega_{\ell^\alpha},\cdot)$ is $C^{1}$ from $PD(A^{\sigma})$ to $\mathcal{C}_{\beta}^{\alpha,-}$. Therefore, $\psi^{\alpha}(\omega_{\ell^\alpha}, \xi)$ is $C^{1}$ in $\xi$. This completes the proof.
\end{proof}\par

Let $\mathcal{C}_{\beta}^{-}$ be a Banach space defined by
\begin{align*}
\mathcal{C}_{\beta}^{-} = \Big\{ \varphi\in C((-\infty, 0], D(A^{\sigma})) \ \big| \sup_{t\leq0} \big\{e^{\beta t - \int_0^t z(\theta_r \omega_{\ell^\alpha}) dr} \|\varphi(t)\|_{\sigma} \big\} < +\infty \Big\}
\end{align*}
with the norm
\begin{align*}
\|\varphi\|_{\mathcal{C}_{\beta}^{-}} = \sup_{t\leq0} \big\{e^{\beta t - \int_0^t z(\theta_r \omega_{\ell^\alpha}) dr} \|\varphi(t)\|_{\sigma}\big\}.
\end{align*}
When $\alpha=2$, by using the same arguments, we can show the similar results for equation \eqref{eq:3-1.3} and \eqref{eq:3-2.5}. Here we summarize them as follows.
\begin{lem}
\label{lem:3-4.2}
$x \in \mathscr{M}(\omega)$ if and only if there exists a function $u(\cdot) \in \mathcal{C}_{\beta}^{-}$ with $u(0) = x$ and satisfies
\begin{align}
\label{3-4.15}
u(t) &= e^{-A t + \int_0^t z(\theta_r \omega) dr} \xi
+ \int_0^t e^{-A(t-s) + \int_s^t z(\theta_r \omega) dr} P G(\theta_s \omega, u(s)) ds \notag\\
&\qquad
+ \int_{-\infty}^t e^{-A(t-s) + \int_s^t z(\theta_r \omega) dr} Q G(\theta_s \omega, u(s)) ds,
\end{align}
where $\xi = P x$.
\end{lem}

\begin{thm}
\label{th:3-4.2}
Assume that the spectral gap condition \eqref{3-4.2}
holds. Then there exists a Lipschitz inertial manifold for equation \eqref{eq:3-2.5}, which is given by
\begin{align*}
 \mathscr{M}(\omega)=\big\lbrace \xi+\psi(\omega, \xi)|\xi\in PD(A^{\sigma})\big\rbrace,
\end{align*}
where $\psi(\omega, \cdot): PD(A^{\sigma})\rightarrow QD(A^{\sigma})$ is Lipschitz continuous defined by
\begin{align*}
\psi(\omega, \xi)= \int_{-\infty}^0e^{As+\int_{s}^{0}z(\theta_{r}\omega)dr}QG(\theta_{s}\omega, \bar{u}(s, \omega, \xi))ds, \ \forall \xi \in PD(A^{\sigma}),
\end{align*}
and $\psi$ is measurable in $(\omega, \xi)$.  Here $\bar{u}(\cdot, \omega, \xi))$ is the unique solution of equation \eqref{3-4.15} in $ \mathcal{C}_{\beta}^{-}$. Furthermore, $\widetilde{\mathscr{M}}(\omega)=T^{-1}(\omega, \mathscr{M}(\omega))$, i.e.,
\begin{align*}
\widetilde{\mathscr{M}}(\omega)=\left\lbrace \xi+e^{z(\omega)}\psi(\omega, e^{-z(\omega)}\xi)|\xi\in PD(A^{\sigma})\right\rbrace
\end{align*}
is the Lipschitz inertial manifold of equation \eqref{eq:3-1.3}.\par
In addition, if $F\in C^{1}$, then $\mathscr{M}(\omega)$ is a $C^{1}$ inertial manifold for equation \eqref{eq:3-2.5}, i.e., $\psi(\omega, \xi)$ is $C^{1}$ in $\xi$. Consequently, $\widetilde{\mathscr{M}}(\omega)$ is a $C^{1}$ inertial manifold of equation \eqref{eq:3-1.3}.
\end{thm}

In what follows, we study the relations between the inertial manifolds of equation \eqref{eq:3-2.4} and equation \eqref{eq:3-2.5}.

\begin{thm}
\label{th:3-4.3}
Assume that the same conditions in Theorem \ref{th:3-4.1} hold. Then we have
\begin{align*}
\|\psi^{\alpha}(\omega_{\ell^\alpha}, \xi)-\psi(\omega, \xi)\|_{\sigma}\rightarrow 0
\end{align*}
in probability as $\alpha\rightarrow 2$. Furthermore,
\begin{align*}
\|D_{\xi}\psi^{\alpha}(\omega_{\ell^\alpha}, \xi)-D_{\xi}\psi(\omega, \xi)\|\rightarrow 0
\end{align*}
in probability as $\alpha\rightarrow 2$.
\end{thm}
\begin{proof}
$\mathbf{Step\ 1}$. We prove that
\begin{align}
\label{3-4.16}
\|u^{\alpha}(\cdot, \omega_{\ell^\alpha}, \xi)-u(\cdot, \omega, \xi)\|_{\mathcal{C}_{\beta}^{\alpha,-}}\rightarrow 0
\end{align}
in probability as $\alpha\rightarrow 2$. Choose $\delta>0$ such that $\lambda_{N}<\beta-j\delta<\beta<\lambda_{N+1}$ and
\begin{align*}
L\left( \frac{\lambda_{N}^{\sigma}}{\beta-j\delta-\lambda_{N}}+\frac{\sigma^{\sigma}\Gamma(1-\sigma)}{(\lambda_{N+1}-\beta+j\delta)^{1-\sigma}}+\frac{\lambda_{N+1}^{\sigma}}{\lambda_{N+1}-\beta+j\delta}\right)
\leq\mu<1, j=1,2,3.
\end{align*}
By Theorem \ref{th:3-4.1} and \ref{th:3-4.2}, this condition yields that there exists $\bar{u}^{\alpha}\in \mathcal{C}_{\beta-j\delta}^{\alpha,-}$ and $\bar{u}\in \mathcal{C}_{\beta-j\delta}^{-}$.
\par

For simplicity, we denote $\bar{u}^{\alpha}(t)=\bar{u}^{\alpha}(t, \omega_{\ell^\alpha}, \xi)$ and $\bar{u}(t)=\bar{u}(t, \omega_{\ell^\alpha}, \xi)$. Let $\bar{z}^{\alpha}(t)=\bar{u}^{\alpha}(t)-\bar{u}(t)$. Similar to Lemma \ref{lem:3-4.1}, we can get
\begin{align*}
\bar{z}^{\alpha}(t)=&\int_0^te^{-A(t-s)+\int_{s}^{t}z(\theta_{r}\omega_{\ell^\alpha})dr}P\Big[G(\theta_{s}\omega_{\ell^\alpha}, \bar{u}^{\alpha}(s))-G(\theta_{s}\omega,\bar{u}(s))\\&\ \ \  +\big(z(\theta_{s}\omega_{\ell^\alpha})-z(\theta_{s}\omega)\big)\bar{u}(s)\Big]ds\\&+\int_{-\infty}^te^{-A(t-s)+\int_{s}^{t}z(\theta_{r}\omega_{\ell^\alpha})dr}Q\Big[G(\theta_{s}\omega_{\ell^\alpha}, \bar{u}^{\alpha}(s))-G(\theta_{s}\omega,\bar{u}(s))\\&\ \ \ \ \ +\big(z(\theta_{s}\omega_{\ell^\alpha})-z(\theta_{s}\omega)\big)\bar{u}(s)\Big]ds\\
=&\mathcal{Z}_{1}+\mathcal{Z}_{2}.
\end{align*}
To estimate these integrals, we find that
\begin{align*}
&\|G(\theta_{s}\omega_{\ell^\alpha}, \bar{u}^{\alpha}(s))-G(\theta_{s}\omega,\bar{u}(s))\|\leq L\|\bar{z}^{\alpha}(s)\|_{\sigma}+2L|e^{z(\theta_{s}\omega_{\ell^\alpha})-z(\theta_{s}\omega)}-1|\|\bar{u}(s)\|_{\sigma}.
\end{align*}
For the first term $\mathcal{Z}_{1}$, we have
\begin{align*}
&\nonumber e^{(\beta-\delta)t-\int_{0}^{t}z(\theta_{r}\omega_{\ell^\alpha})dr}\|\mathcal{Z}_{1}\|_{\sigma}\\&\nonumber\leq L\lambda_{N}^{\sigma}\int_t^0e^{-\lambda_{N}(t-s)+(\beta-\delta)(t-s)}\|\bar{z}^{\alpha}\|_{\mathcal{C}_{\beta-\delta}^{\alpha,-}}ds \nonumber
\\&\ \ \ +2L\lambda_{N}^{\sigma}\int_t^0e^{-\lambda_{N}(t-s)+(\beta-\delta)(t-s)}e^{\delta s}|e^{z(\theta_{s}\omega_{\ell^\alpha})-z(\theta_{s}\omega)}-1| \|\bar{u}\|_{\mathcal{C}_{\beta-2\delta}^{\alpha,-}}ds\nonumber
\\&\ \ \ +\lambda_{N}^{\sigma}\int_t^0e^{-\lambda_{N}(t-s)+(\beta-\delta)(t-s)}e^{\delta s}|z(\theta_{s}\omega_{\ell^\alpha})-z(\theta_{s}\omega)| \|\bar{u}\|_{\mathcal{C}_{\beta-2\delta}^{\alpha,-}}ds\nonumber
\\&\nonumber\leq \frac{L\lambda_{N}^{\sigma}}{\beta-\delta-\lambda_{N}}\|\bar{z}^{\alpha}\|_{\mathcal{C}_{\beta-\delta}^{\alpha,-}}
+2L\lambda_{N}^{\sigma}\|\bar{u}\|_{\mathcal{C}_{\beta-2\delta}^{\alpha,-}}\int_t^0e^{\delta s}|e^{z(\theta_{s}\omega_{\ell^\alpha})-z(\theta_{s}\omega)}-1|ds\nonumber
\\&\ \ \
+\lambda_{N}^{\sigma}\|\bar{u}\|_{\mathcal{C}_{\beta-2\delta}^{\alpha,-}}\int_t^0e^{\delta s}|z(\theta_{s}\omega_{\ell^\alpha})-z(\theta_{s}\omega)|ds,
\end{align*}
which implies that
\begin{align}
\label{3-4.17}
\|\mathcal{Z}_{1}\|_{\mathcal{C}_{\beta-\delta}^{\alpha,-}}&\leq \frac{L\lambda_{N}^{\sigma}}{\beta-\delta-\lambda_{N}}\|\bar{z}^{\alpha}\|_{\mathcal{C}_{\beta-\delta}^{\alpha,-}}
+2L\lambda_{N}^{\sigma}\|\bar{u}\|_{\mathcal{C}_{\beta-3\delta}^{-}}\int_{t}^0e^{\delta s}|e^{z(\theta_{s}\omega_{\ell^\alpha})-z(\theta_{s}\omega)}-1|ds\nonumber
\\&\ \ \
+\lambda_{N}^{\sigma}\|\bar{u}\|_{\mathcal{C}_{\beta-3\delta}^{-}}\int_{t}^0e^{\delta s}|z(\theta_{s}\omega_{\ell^\alpha})-z(\theta_{s}\omega)|ds.
\end{align}
For the second term $\mathcal{Z}_{2}$, we get
\begin{align}
\label{3-4.18}
&\nonumber e^{(\beta-\delta)t-\int_{0}^{t}z(\theta_{r}\omega_{\ell^\alpha})dr}\|\mathcal{Z}_{2}\|_{\sigma}\\&\nonumber\leq L\int_{-\infty}^t \Big[ \Big( \frac{\sigma}{t-s} \Big)^{\sigma} + \lambda_{N+1}^{\sigma} \Big] e^{(\beta-\delta-\lambda_{N+1})(t-s)} \|\bar{z}^{\alpha}\|_{\mathcal{C}_{\beta-\delta}^{\alpha,-}}ds \nonumber
\\&\ \ \ +2L\int_{-\infty}^t \Big[ \Big( \frac{\sigma}{t-s} \Big)^{\sigma} + \lambda_{N+1}^{\sigma} \Big] e^{(\beta-\delta-\lambda_{N+1})(t-s)}e^{\delta s}|e^{z(\theta_{s}\omega_{\ell^\alpha})-z(\theta_{s}\omega)}-1| \|\bar{u}\|_{\mathcal{C}_{\beta-2\delta}^{\alpha,-}}ds\nonumber
\\&\ \ \ +\int_{-\infty}^t \Big[ \Big( \frac{\sigma}{t-s} \Big)^{\sigma} + \lambda_{N+1}^{\sigma} \Big] e^{(\beta-\delta-\lambda_{N+1})(t-s)}e^{\delta s}|z(\theta_{s}\omega_{\ell^\alpha})-z(\theta_{s}\omega)| \|\bar{u}\|_{\mathcal{C}_{\beta-2\delta}^{\alpha,-}}ds\nonumber
\\& =\mathcal{Z}_{21}+\mathcal{Z}_{22}+\mathcal{Z}_{23}
\end{align}
To deal with $\mathcal{Z}_{2}$, we estimate each of the above three integrals. We first estimate
\begin{align}
\label{3-4.19}
\mathcal{Z}_{21}&=L\int_{-\infty}^t \Big[ \Big( \frac{\sigma}{t-s} \Big)^{\sigma} + \lambda_{N+1}^{\sigma} \Big] e^{(\beta-\delta-\lambda_{N+1})(t-s)} \|\bar{z}^{\alpha}\|_{\mathcal{C}_{\beta-\delta}^{\alpha,-}}ds \nonumber
\\&
\leq L\left(\frac{\sigma^{\sigma}\Gamma(1-\sigma)}{(\lambda_{N+1}-\beta+\delta)^{1-\sigma}}
+\frac{\lambda_{N+1}^{\sigma}}{\lambda_{N+1}-\beta+\delta}\right)\|\bar{z}^{\alpha}\|_{\mathcal{C}_{\beta-\delta}^{\alpha,-}}.
\end{align}
For the second integral $\mathcal{Z}_{22}$, we have
\begin{align}
\label{3-4.20}
\mathcal{Z}_{22}&=2L\int_{-\infty}^t \Big[ \Big( \frac{\sigma}{t-s} \Big)^{\sigma} + \lambda_{N+1}^{\sigma} \Big] e^{(\beta-\delta-\lambda_{N+1})(t-s)}e^{\delta s}|e^{z(\theta_{s}\omega_{\ell^\alpha})-z(\theta_{s}\omega)}-1| \|\bar{u}\|_{\mathcal{C}_{\beta-2\delta}^{\alpha,-}}ds\nonumber
\\&
\leq 2L \sigma^{\sigma}\|\bar{u}\|_{\mathcal{C}_{\beta-3\delta}^{-}}\int_{-\infty}^t (t-s)^{-\sigma} e^{(\beta-\delta-\lambda_{N+1})(t-s)}e^{\delta s}|e^{z(\theta_{s}\omega_{\ell^\alpha})-z(\theta_{s}\omega)}-1|ds\nonumber
\\&\ \ \ +2L\lambda_{N+1}^{\sigma}\|\bar{u}\|_{\mathcal{C}_{\beta-3\delta}^{-}}\int_{-\infty}^te^{\delta s}|e^{z(\theta_{s}\omega_{\ell^\alpha})-z(\theta_{s}\omega)}-1|ds.
\end{align}
Similarly, we have
\begin{align}
\label{3-4.21}
\mathcal{Z}_{23}&=\int_{-\infty}^t \Big[ \Big( \frac{\sigma}{t-s} \Big)^{\sigma} + \lambda_{N+1}^{\sigma} \Big] e^{(\beta-\delta-\lambda_{N+1})(t-s)}e^{\delta s}|z(\theta_{s}\omega_{\ell^\alpha})-z(\theta_{s}\omega)| \|\bar{u}\|_{\mathcal{C}_{\beta-2\delta}^{\alpha,-}}ds\nonumber
\\&
\leq \sigma^{\sigma}\|\bar{u}\|_{\mathcal{C}_{\beta-3\delta}^{-}}\int_{-\infty}^t (t-s)^{-\sigma} e^{(\beta-\delta-\lambda_{N+1})(t-s)}e^{\delta s}|z(\theta_{s}\omega_{\ell^\alpha})-z(\theta_{s}\omega)|ds\nonumber
\\&\ \ \ +\lambda_{N+1}^{\sigma}\|\bar{u}\|_{\mathcal{C}_{\beta-3\delta}^{-}}\int_{-\infty}^te^{\delta s}|z(\theta_{s}\omega_{\ell^\alpha})-z(\theta_{s}\omega)|ds.
\end{align}
Combining with \eqref{3-4.17}-\eqref{3-4.21}, we obtain
\begin{align}
\label{3-4.22}
\|\bar{z}^{\alpha}\|_{\mathcal{C}_{\beta-\delta}^{\alpha,-}}&\leq \mu\|\bar{z}^{\alpha}\|_{\mathcal{C}_{\beta-\delta}^{\alpha,-}} +2L\lambda_{N+1}^{\sigma}\|\bar{u}\|_{\mathcal{C}_{\beta-3\delta}^{-}}\int_{-\infty}^0e^{\delta s}|e^{z(\theta_{s}\omega_{\ell^\alpha})-z(\theta_{s}\omega)}-1|ds\nonumber
\\&\ \ \
+\lambda_{N+1}^{\sigma}\|\bar{u}\|_{\mathcal{C}_{\beta-3\delta}^{-}}\int_{-\infty}^0e^{\delta s}|z(\theta_{s}\omega_{\ell^\alpha})-z(\theta_{s}\omega)|ds\nonumber
\\&\ \ \
+2L \sigma^{\sigma}\|\bar{u}\|_{\mathcal{C}_{\beta-3\delta}^{-}}\int_{-\infty}^t (t-s)^{-\sigma} e^{(\beta-\delta-\lambda_{N+1})(t-s)}e^{\delta s}|e^{z(\theta_{s}\omega_{\ell^\alpha})-z(\theta_{s}\omega)}-1|ds\nonumber
\\&\ \ \
+\sigma^{\sigma}\|\bar{u}\|_{\mathcal{C}_{\beta-3\delta}^{-}}\int_{-\infty}^t (t-s)^{-\sigma} e^{(\beta-\delta-\lambda_{N+1})(t-s)}e^{\delta s}|z(\theta_{s}\omega_{\ell^\alpha})-z(\theta_{s}\omega)|ds.
\end{align}\par
As $0<\sigma<1$, we choose $0<q_{1}<\frac{1}{\sigma}$ and let $q_{2}$ satisfies $ \frac{1}{q_{1}}+\frac{1}{q_{2}}=1$. Using the  H\"{o}lder inequality, we get
\begin{align}
\label{3-4.23}
&\int_{-\infty}^t (t-s)^{-\sigma} e^{(\beta-\delta-\lambda_{N+1})(t-s)}e^{\delta s}|e^{z(\theta_{s}\omega_{\ell^\alpha})-z(\theta_{s}\omega)}-1|ds\nonumber
\\&\leq\left\{\int_{-\infty}^t (t-s)^{-\sigma q_{1}} e^{q_{1}(\beta-\delta-\lambda_{N+1})(t-s)} ds \right\}^{\frac{1}{q_{1}}}
\left\{\int_{-\infty}^t e^{q_{2}\delta s}
|e^{z(\theta_{s}\omega_{\ell^\alpha})-z(\theta_{s}\omega)}-1|^{q_{2}} ds \right\}^{\frac{1}{q_{2}}}\nonumber
\\&\leq\left\{ \frac{ \Gamma(1-\sigma q_{1}) }{ [q_{1}(\lambda_{N+1}-\beta+\delta)]^{1-\sigma q_{1}} } \right\}^{\frac{1}{q_{1}}}\left\{\int_{-\infty}^0 e^{q_{2}\delta s}
|e^{z(\theta_{s}\omega_{\ell^\alpha})-z(\theta_{s}\omega)}-1|^{q_{2}} ds \right\}^{\frac{1}{q_{2}}},
\end{align}
and
\begin{align}
\label{3-4.24}
&\int_{-\infty}^t (t-s)^{-\sigma} e^{(\beta-\delta-\lambda_{N+1})(t-s)}e^{\delta s}|z(\theta_{s}\omega_{\ell^\alpha})-z(\theta_{s}\omega)|ds\nonumber
\\&\leq\left\{\int_{-\infty}^t (t-s)^{-\sigma q_{1}} e^{q_{1}(\beta-\delta-\lambda_{N+1})(t-s)} ds \right\}^{\frac{1}{q_{1}}}
\left\{\int_{-\infty}^t e^{q_{2}\delta s}
|z(\theta_{s}\omega_{\ell^\alpha})-z(\theta_{s}\omega)|^{q_{2}} ds \right\}^{\frac{1}{q_{2}}}\nonumber
\\&\leq\left\{ \frac{ \Gamma(1-\sigma q_{1}) }{ [q_{1}(\lambda_{N+1}-\beta+\delta)]^{1-\sigma q_{1}} } \right\}^{\frac{1}{q_{1}}}\left\{\int_{-\infty}^0 e^{q_{2}\delta s}
|z(\theta_{s}\omega_{\ell^\alpha})-z(\theta_{s}\omega)|^{q_{2}} ds \right\}^{\frac{1}{q_{2}}},
\end{align}
As $\sigma=0$, we have
\begin{align}
\label{3-4.23a}
&\int_{-\infty}^t (t-s)^{-\sigma} e^{(\beta-\delta-\lambda_{N+1})(t-s)}e^{\delta s}|e^{z(\theta_{s}\omega_{\ell^\alpha})-z(\theta_{s}\omega)}-1|ds\nonumber
\\&\leq \int_{-\infty}^0 e^{\delta s}
|e^{z(\theta_{s}\omega_{\ell^\alpha})-z(\theta_{s}\omega)}-1| ds,
\end{align}
and
\begin{align}
\label{3-4.24a}
&\int_{-\infty}^t (t-s)^{-\sigma} e^{(\beta-\delta-\lambda_{N+1})(t-s)}e^{\delta s}|z(\theta_{s}\omega_{\ell^\alpha})-z(\theta_{s}\omega)|ds\nonumber
\\&\leq\int_{-\infty}^0 e^{\delta s}
|z(\theta_{s}\omega_{\ell^\alpha})-z(\theta_{s}\omega)| ds ,
\end{align}\par
By Lemma \ref{lem:3-2.2}, we have for each subsequence $\{\alpha_n\}_{n \in \mathbb{N}}$, there exists a sub-subsequence $\{\alpha_{n(k)}\}_{k \in \mathbb{N}}$ such that $z(\theta_s \omega_{\ell^{\alpha_{n(k)}}})$ converges to $z(\theta_s \omega)$ in $D_u([-T,0], \mathbb{R})$, almost surely, i.e., for every $T > 0$,
\begin{align*}
\lim_{k \to \infty} \sup_{s \in [-T,0]}  | z(\theta_s \omega_{\ell^{\alpha_{n(k)}}}) - z(\theta_s \omega) | = 0,
\end{align*}
which implies that for every $s < 0$,
\begin{align*}
| z(\theta_s \omega_{\ell^{\alpha_{n(k)}}}) - z(\theta_s \omega) | \to 0 \ \  \text{almost surely}.
\end{align*}
Along with Lemma \ref{lem:3-2.1} and the dominated convergence theorem, we have
\begin{align*}
\int_{-\infty}^0 e^{\delta s} \Big[|e^{z(\theta_s \omega_{\ell^{\alpha_{n(k)}}})- z(\theta_s \omega)} -1|+ | z(\theta_s \omega_{\ell^{\alpha_{n(k)}}}) - z(\theta_s \omega) |\Big]ds \to 0,
\end{align*}
almost surely. Thus
\begin{align}
\label{3-4.25}
\int_{-\infty}^0 e^{\delta s} \Big[ | e^{z(\theta_s \omega_{\ell^{\alpha}})- z(\theta_s \omega)} -1|+ | z(\theta_s \omega_{\ell^{\alpha}}) - z(\theta_s \omega) |\Big]ds \to 0
\end{align}
in probability as $\alpha\rightarrow2$. Moreover, we have \begin{align}
\label{3-4.25a}
\int_{-\infty}^0 e^{q_{2}\delta s} \Big[ | e^{z(\theta_s \omega_{\ell^{\alpha}})- z(\theta_s \omega)} -1|^{q_{2}}+ | z(\theta_s \omega_{\ell^{\alpha}}) - z(\theta_s \omega) |^{q_{2}}\Big]ds \to 0
\end{align}
in probability as $\alpha\rightarrow2$. Together with \eqref{3-4.22}-\eqref{3-4.25a}, we obtain that
\begin{align*}
\|\bar{z}^{\alpha}\|_{\mathcal{C}_{\beta}^{\alpha,-}}\leq\|\bar{z}^{\alpha}\|_{\mathcal{C}_{\beta-\delta}^{\alpha,-}} \rightarrow 0
\end{align*}
in probability as $\alpha\rightarrow2$. Therefore,
\begin{align*}
\|\bar{u}^{\alpha}(\cdot, \omega_{\ell^\alpha}, \xi)-\bar{u}(\cdot, \omega, \xi)\|_{\mathcal{C}_{\beta}^{\alpha,-}}\rightarrow 0
\end{align*}
in probability as $\alpha\rightarrow2$.\\
$\mathbf{Step\ 2}$. We show that
\begin{align}
\label{3-4.26}
\|D_{\xi}\bar{u}^{\alpha}(\cdot, \omega_{\ell^\alpha}, \xi)-D_{\xi}\bar{u}(\cdot, \omega, \xi)\|_{\mathcal{L}(PD(A^{\sigma}), \mathcal{C}_{\beta}^{\alpha,-})}\rightarrow 0
\end{align}
in probability as $\alpha\rightarrow2$.\par
Note that for $t\leq0$,
\begin{align*}
D_{\xi}\bar{z}^{\alpha}(t)=&\int_0^te^{-A(t-s)+\int_{0}^{t}z(\theta_{r}\omega_{\ell^\alpha})dr}P\Big[ D_{u}G(\theta_{s}\omega_{\ell^\alpha},\bar{u}^{\alpha}(s))D_{\xi}\bar{z}^{\alpha}(s)\\&\ \ \ +\big(D_{u}G(\theta_{s}\omega_{\ell^\alpha},\bar{u}^{\alpha}(s))-D_{u}G(\theta_{s}\omega,\bar{u}(s))\big) D_{\xi}\bar{u}(s)\\&\ \ \ +(z(\theta_{s}\omega_{\ell^\alpha})-z(\theta_{s}\omega))D_{\xi}\bar{u}(s)\Big]ds\\& +\int_{-\infty}^te^{-A(t-s)+\int_{0}^{t}z(\theta_{r}\omega_{\ell^\alpha})dr}Q\Big[ D_{u}G(\theta_{s}\omega_{\ell^\alpha},\bar{u}^{\alpha}(s))D_{\xi}\bar{z}^{\alpha}(s)\\&\ \ \ +\big(D_{u}G(\theta_{s}\omega_{\ell^\alpha},\bar{u}^{\alpha}(s))-D_{u}G(\theta_{s}\omega,\bar{u}(s))\big) D_{\xi}\bar{u}(s)\\&\ \ \ +(z(\theta_{s}\omega_{\ell^\alpha})-z(\theta_{s}\omega))D_{\xi}\bar{u}(s)\Big]ds.
\end{align*}
We denote the above two integrals by $ \mathcal{I}_{1}$ and $\mathcal{I}_{2}$, respectively. We first estimate $ \mathcal{I}_{1}$,
\begin{align}
\label{3-4.27}
&\nonumber e^{(\beta-\delta)t-\int_{0}^{t}z(\theta_{r}\omega_{\ell^\alpha})dr}\|\mathcal{I}_{1}\|_{\sigma}\\&\nonumber\leq L\lambda_{N}^{\sigma}\|D_{\xi}\bar{z}^{\alpha}\|_{\mathcal{L}(PD(A^{\sigma}), \mathcal{C}_{\beta-\delta}^{\alpha,-})}\int_t^0e^{(\beta-\delta-\lambda_{N})(t-s)}ds\nonumber
\\&\ \ \ +\lambda_{N}^{\sigma}\|D_{\xi}\bar{u}\|_{\mathcal{L}(PD(A^{\sigma}), \mathcal{C}_{\beta-2\delta}^{\alpha,-})}\int_t^0e^{(\beta-\delta-\lambda_{N})(t-s)}e^{\delta s}\nonumber
\\&\ \ \ \ \ \ \ \ \times\|D_{u}F(e^{z(\theta_{t}\omega_{\ell^\alpha})}\bar{u}^{\alpha}(s))-D_{u}F(e^{z(\theta_{t}\omega)}\bar{u}(s))\|ds\nonumber
\\&\ \ \ \nonumber +\|D_{\xi}\bar{u}\|_{\mathcal{L}(PD(A^{\sigma}), \mathcal{C}_{\beta-2\delta}^{\alpha,-})}\int_t^0e^{(\beta-\delta-\lambda_{N})(t-s)}e^{\delta s}|z(\theta_{s}\omega_{\ell^\alpha})-z(\theta_{s}\omega)|ds
\\&\nonumber\leq \frac{L\lambda_{N}^{\sigma}}{\beta-\delta-\lambda_{N}}\|D_{\xi}\bar{z}^{\alpha}\|_{\mathcal{L}(PD(A^{\sigma}), \mathcal{C}_{\beta-\delta}^{\alpha,-})}\nonumber
\\&\ \ \ +\lambda_{N}^{\sigma}\|D_{\xi}\bar{u}\|_{\mathcal{L}(PD(A^{\sigma}), \mathcal{C}_{\beta-3\delta}^{-})}\int_t^0e^{\delta s} \|D_{u}F(e^{z(\theta_{t}\omega_{\ell^\alpha})}\bar{u}^{\alpha}(s))-D_{u}F(e^{z(\theta_{t}\omega)}\bar{u}(s))\|ds\nonumber
\\&\ \ \  +\lambda_{N}^{\sigma}\|D_{\xi}\bar{u}\|_{\mathcal{L}(PD(A^{\sigma}), \mathcal{C}_{\beta-3\delta}^{-})}\int_t^0e^{\delta s} |z(\theta_{s}\omega_{\ell^\alpha})-z(\theta_{s}\omega)|ds.
\end{align}
Similarly, by \eqref{3-2.1} and \eqref{3-2.2}, we obtain
\begin{align}
\label{3-4.28}
&\nonumber e^{(\beta-\delta)t-\int_{0}^{t}z(\theta_{r}\omega_{\ell^\alpha})dr}\|\mathcal{I}_{2}\|_{\sigma}\\&\nonumber\leq L\|D_{\xi}\bar{z}^{\alpha}\|_{\mathcal{L}(PD(A^{\sigma}), \mathcal{C}_{\beta-\delta}^{\alpha,-})}\int_{-\infty}^t \Big[ \Big( \frac{\sigma}{t-s} \Big)^{\sigma} + \lambda_{N+1}^{\sigma} \Big] e^{(\beta-\delta-\lambda_{N+1})(t-s)} ds \nonumber
\\&\ \ \ +\|D_{\xi}\bar{u}\|_{\mathcal{L}(PD(A^{\sigma}), \mathcal{C}_{\beta-2\delta}^{\alpha,-})}\int_{-\infty}^t \Big[ \Big( \frac{\sigma}{t-s} \Big)^{\sigma} + \lambda_{N+1}^{\sigma} \Big] e^{(\beta-\delta-\lambda_{N+1})(t-s)}e^{\delta s}\nonumber
\\&\ \ \ \ \ \ \ \ \times\|D_{u}F(e^{z(\theta_{t}\omega_{\ell^\alpha})}\bar{u}^{\alpha}(s))-D_{u}F(e^{z(\theta_{t}\omega)}\bar{u}(s))\|ds\nonumber
\\&\ \ \ +\|D_{\xi}\bar{u}\|_{\mathcal{L}(PD(A^{\sigma}), \mathcal{C}_{\beta-2\delta}^{\alpha,-})}\int_{-\infty}^t \Big[ \Big( \frac{\sigma}{t-s} \Big)^{\sigma} + \lambda_{N+1}^{\sigma} \Big] e^{(\beta-\delta-\lambda_{N+1})(t-s)}e^{\delta s}|z(\theta_{s}\omega_{\ell^\alpha})-z(\theta_{s}\omega)| ds\nonumber
\\& =\mathcal{I}_{21}+\mathcal{I}_{22}+\mathcal{I}_{23}.
\end{align}
For the first integral $\mathcal{I}_{21}$, we find
\begin{align}
\label{3-4.29}
\mathcal{I}_{21}\leq L\left(\frac{\sigma^{\sigma}\Gamma(1-\sigma)}{(\lambda_{N+1}-\beta+\delta)^{1-\sigma}}
+\frac{\lambda_{N+1}^{\sigma}}{\lambda_{N+1}-\beta+\delta}\right)\|D_{\xi}\bar{z}^{\alpha}\|_{\mathcal{L}(PD(A^{\sigma}), \mathcal{C}_{\beta-\delta}^{\alpha,-})}.
\end{align}
We now estimate the second term $\mathcal{I}_{22}$,
\begin{align}
\label{3-4.30}
\mathcal{I}_{22}&\leq \sigma^{\sigma}\|D_{\xi}\bar{u}\|_{\mathcal{L}(PD(A^{\sigma}), \mathcal{C}_{\beta-3\delta}^{-})}\int_{-\infty}^t (t-s)^{-\sigma} e^{(\beta-\delta-\lambda_{N+1})(t-s)}e^{\delta s}\nonumber
\\&\ \ \ \ \ \ \ \ \times\|D_{u}F(e^{z(\theta_{t}\omega_{\ell^\alpha})}\bar{u}^{\alpha}(s))-D_{u}F(e^{z(\theta_{t}\omega)}\bar{u}(s))\|ds\nonumber
\\&\ \ \ +\lambda_{N+1}^{\sigma}\|D_{\xi}\bar{u}\|_{\mathcal{L}(PD(A^{\sigma}), \mathcal{C}_{\beta-3\delta}^{-})}\int_{-\infty}^te^{\delta s}\nonumber
\\&\ \ \ \ \ \ \ \ \times\|D_{u}F(e^{z(\theta_{t}\omega_{\ell^\alpha})}\bar{u}^{\alpha}(s))-D_{u}F(e^{z(\theta_{t}\omega)}\bar{u}(s))\|ds.
\end{align}
Using the same arguments for $\mathcal{I}_{22}$, we have
\begin{align}
\label{3-4.31}
\mathcal{I}_{23}&\leq \sigma^{\sigma}\|D_{\xi}\bar{u}\|_{\mathcal{L}(PD(A^{\sigma}), \mathcal{C}_{\beta-3\delta}^{-})}\int_{-\infty}^t (t-s)^{-\sigma} e^{(\beta-\delta-\lambda_{N+1})(t-s)}e^{\delta s}|z(\theta_{s}\omega_{\ell^\alpha})-z(\theta_{s}\omega)|ds\nonumber
\\&\ \ \ +\lambda_{N+1}^{\sigma}\|D_{\xi}\bar{u}\|_{\mathcal{L}(PD(A^{\sigma}), \mathcal{C}_{\beta-3\delta}^{-})}\int_{-\infty}^te^{\delta s}|z(\theta_{s}\omega_{\ell^\alpha})-z(\theta_{s}\omega)|ds.
\end{align}
Combining \eqref{3-4.27}-\eqref{3-4.31}, we get
\begin{align}
\label{3-4.32}
&\|D_{\xi}\bar{z}^{\alpha}(t)\|_{\mathcal{L}(PD(A^{\sigma}), \mathcal{C}_{\beta-\delta}^{\alpha,-})}\nonumber
\\&\leq \mu\|D_{\xi}\bar{z}^{\alpha}(t)\|_{\mathcal{L}(PD(A^{\sigma}), \mathcal{C}_{\beta-\delta}^{\alpha,-})} +\lambda_{N+1}^{\sigma}\|D_{\xi}\bar{u}\|_{\mathcal{L}(PD(A^{\sigma}), \mathcal{C}_{\beta-3\delta}^{-})}\int_{-\infty}^0e^{\delta s}\nonumber
\\&\ \ \ \ \ \ \ \ \times\|D_{u}F(e^{z(\theta_{t}\omega_{\ell^\alpha})}\bar{u}^{\alpha}(s))-D_{u}F(e^{z(\theta_{t}\omega)}\bar{u}(s))\|ds\nonumber
\\&\ \ \
+\lambda_{N+1}^{\sigma}\|D_{\xi}\bar{u}\|_{\mathcal{L}(PD(A^{\sigma}), \mathcal{C}_{\beta-3\delta}^{-})}\int_{-\infty}^0e^{\delta s}|z(\theta_{s}\omega_{\ell^\alpha})-z(\theta_{s}\omega)|ds\nonumber
\\&\ \ \
+\sigma^{\sigma}\|D_{\xi}\bar{u}\|_{\mathcal{L}(PD(A^{\sigma}), \mathcal{C}_{\beta-3\delta}^{-})}\int_{-\infty}^t (t-s)^{-\sigma} e^{(\beta-\delta-\lambda_{N+1})(t-s)}e^{\delta s}\nonumber
\\&\ \ \ \ \ \ \ \ \times\|D_{u}F(e^{z(\theta_{t}\omega_{\ell^\alpha})}\bar{u}^{\alpha}(s))-D_{u}F(e^{z(\theta_{t}\omega)}\bar{u}(s))\|ds\nonumber
\\&\ \ \
+\sigma^{\sigma}\|D_{\xi}\bar{u}\|_{\mathcal{L}(PD(A^{\sigma}), \mathcal{C}_{\beta-3\delta}^{-})}\int_{-\infty}^t (t-s)^{-\sigma} e^{(\beta-\delta-\lambda_{N+1})(t-s)}e^{\delta s}|z(\theta_{s}\omega_{\ell^\alpha})-z(\theta_{s}\omega)|ds.
\end{align}
As \eqref{3-4.23} and \eqref{3-4.24}, for $0<\sigma<1$, by using the  H\"{o}lder inequality, we get
\begin{align}
\label{3-4.33}
&\int_{-\infty}^t (t-s)^{-\sigma} e^{(\beta-\delta-\lambda_{N+1})(t-s)}e^{\delta s}\|D_{u}F(e^{z(\theta_{t}\omega_{\ell^\alpha})}\bar{u}^{\alpha}(s))-D_{u}F(e^{z(\theta_{t}\omega)}\bar{u}(s))\|ds\nonumber
\\&\leq\left\{\int_{-\infty}^t (t-s)^{-\sigma q_{1}} e^{q_{1}(\beta-\delta-\lambda_{N+1})(t-s)} ds \right\}^{\frac{1}{q_{1}}}\nonumber
\\&\ \ \ \ \ \ \ \ \times
\left\{\int_{-\infty}^t e^{q_{2}\delta s}
\|D_{u}F(e^{z(\theta_{t}\omega_{\ell^\alpha})}\bar{u}^{\alpha}(s))-D_{u}F(e^{z(\theta_{t}\omega)}\bar{u}(s))\|^{q_{2}} ds \right\}^{\frac{1}{q_{2}}}\nonumber
\\&\leq\left\{ \frac{ \Gamma(1-\sigma q_{1}) }{ [q_{1}(\lambda_{N+1}-\beta+\delta)]^{1-\sigma q_{1}} } \right\}^{\frac{1}{q_{1}}}\nonumber
\\&\ \ \ \ \ \ \ \ \times\left\{\int_{-\infty}^0 e^{q_{2}\delta s}
\|D_{u}F(e^{z(\theta_{t}\omega_{\ell^\alpha})}\bar{u}^{\alpha}(s))-D_{u}F(e^{z(\theta_{t}\omega)}\bar{u}(s))\|^{q_{2}} ds \right\}^{\frac{1}{q_{2}}}.
\end{align}
For $\sigma=0$, we have
\begin{align}
\label{3-4.33a}
&\int_{-\infty}^t (t-s)^{-\sigma} e^{(\beta-\delta-\lambda_{N+1})(t-s)}e^{\delta s}\|D_{u}F(e^{z(\theta_{t}\omega_{\ell^\alpha})}\bar{u}^{\alpha}(s))-D_{u}F(e^{z(\theta_{t}\omega)}\bar{u}(s))\|ds\nonumber
\\&\leq\int_{-\infty}^0 e^{\delta s}
\|D_{u}F(e^{z(\theta_{t}\omega_{\ell^\alpha})}\bar{u}^{\alpha}(s))-D_{u}F(e^{z(\theta_{t}\omega)}\bar{u}(s))\| ds .
\end{align}
\par
According to the discussion in Step\ 1, for every $s \in (-\infty,0]$ and the sub-subsequence $\{\alpha_{n(k)}\}_{k \in \mathbb{N}}$, we have
\begin{align*}
e^{z(\theta_s \omega_{\ell^{\alpha_{n(k)}}})}\bar{u}^{\alpha_{n(k)}}(s)\rightarrow  e^{z(\theta_{t}\omega)}\bar{u}(s)  \ \  \text{almost surely}.
\end{align*}
It follows from the continuity of $D_{u}F$  that
\begin{align*}
\|D_{u}F(e^{z(\theta_s \omega_{\ell^{\alpha_{n(k)}}})}\bar{u}^{\alpha_{n(k)}}(s))-D_{u}F(e^{z(\theta_{t}\omega)}\bar{u}(s))\| \to 0 \ \ \text{almost surely}.
\end{align*}
By the dominated convergence theorem, we get
\begin{align*}
\int^{+\infty}_0e^{-\sigma s}\|D_{u}F(e^{z(\theta_s \omega_{\ell^{\alpha_{n(k)}}})}\bar{u}^{\alpha_{n(k)}}(s))-D_{u}F(e^{z(\theta_{t}\omega)}\bar{u}(s))\|ds \to 0 \ \  \text{almost surely}.
\end{align*}
This implies that
\begin{align}
\label{3-4.34}
\int_{-\infty}^0 e^{\delta s}
\|D_{u}F(e^{z(\theta_{t}\omega_{\ell^\alpha})}\bar{u}^{\alpha}(s))-D_{u}F(e^{z(\theta_{t}\omega)}\bar{u}(s))\| ds \to 0
\end{align}
in probability as $\alpha\rightarrow2$. Based on the above analysis, we can also conclude that
\begin{align}
\label{3-4.35}
\int_{-\infty}^0 e^{q_{2}\delta s}
\|D_{u}F(e^{z(\theta_{t}\omega_{\ell^\alpha})}\bar{u}^{\alpha}(s))-D_{u}F(e^{z(\theta_{t}\omega)}\bar{u}(s))\|^{q_{2}} ds \to 0
\end{align}
in probability as $\alpha\rightarrow2$.
\par
Together with  \eqref{3-4.24}, \eqref{3-4.24a}-\eqref{3-4.25} and \eqref{3-4.32}-\eqref{3-4.35}, we obtain
\begin{align*}
\|D_{\xi}\bar{z}^{\alpha}(t)\|_{\mathcal{L}(PD(A^{\sigma}), \mathcal{C}_{\beta}^{\alpha,-})}\leq \|D_{\xi}\bar{z}^{\alpha}(t)\|_{\mathcal{L}(PD(A^{\sigma}), \mathcal{C}_{\beta-\delta}^{\alpha,-})} \to 0
\end{align*}
in probability as $\alpha\rightarrow2$. Therefore
\begin{align*}
\|D_{\xi}\bar{u}^{\alpha}(\cdot, \omega_{\ell^{\alpha}}, \xi)-D_{\xi}\bar{u}(\cdot, \omega, \xi)\|_{\mathcal{L}(PD(A^{\sigma}), \mathcal{C}_{\beta}^{\alpha,-})}\to 0
\end{align*}
in probability as $\alpha\rightarrow2$. \\
$\mathbf{Step\ 3}$. We prove that the convergence of $\psi^{\alpha}(\omega_{\ell^\alpha}, \xi)$ and $D_{\xi}\psi^{\alpha}(\omega_{\ell^\alpha}, \xi)$.\par
Since $\psi^{\alpha}(\omega_{\ell^\alpha}, \xi)=Q\bar{u}^{\alpha}(0, \omega_{\ell^\alpha}, \xi)$, by \eqref{3-4.16} and \eqref{3-4.26}, we have
\begin{align*}
\|\psi^{\alpha}(\omega_{\ell^\alpha}, \xi)-\psi(\omega, \xi)\|\leq& \|Q\|\cdot\|\bar{u}^{\alpha}(\cdot, \omega_{\ell^\alpha}, \xi)-\bar{u}(\cdot, \omega, \xi)\|_{\mathcal{C}_{\beta}^{\alpha,-}}\to 0
\end{align*}
in probability as $\alpha\rightarrow 2$, and
\begin{align*}
\|D_{\xi}\psi^{\alpha}(\omega_{\ell^\alpha}, \xi)-D_{\xi}\psi(\omega, \xi)\|\leq& \|Q\|\cdot\|D_{\xi}\bar{u}^{\alpha}(\cdot, \omega_{\ell^\alpha}, \xi)-D_{\xi}\bar{u}(\cdot, \omega, \xi)\|_{\mathcal{L}(PD(A^{\sigma}), \mathcal{C}_{\beta}^{\alpha,-})}\to 0
\end{align*}
in probability as $\alpha\rightarrow 2$. This completes the proof.
\end{proof}
\par
The following theorem shows that the convergence of $\widetilde{\mathscr{M}}^{\alpha}(\omega_{\ell^\alpha})$ and $\widetilde{\mathscr{M}}(\omega)$ as $\alpha$ tends to $2$.

\begin{thm}
\label{th: 3-4.4}
Assume that the same conditions in Theorem \ref{th:3-4.1} hold. Then
\begin{align*}
\|e^{z(\omega_{\ell^\alpha})}\psi^{\alpha}(\omega_{\ell^\alpha}, e^{-z(\omega_{\ell^\alpha})}\xi)- e^{z(\omega)}\psi(\omega, e^{-z(\omega)}\xi)\|_{\sigma} \rightarrow 0
\end{align*}
in probability as $\alpha\rightarrow 2$. Furthermore,
\begin{align*}
\|D_{\xi}\big(e^{z(\omega_{\ell^\alpha})}\psi^{\alpha}(\omega_{\ell^\alpha}, e^{-z(\omega_{\ell^\alpha})}\xi)\big)- D_{\xi}\big(e^{z(\omega)}\psi(\omega, e^{-z(\omega)}\xi)\big)\|\rightarrow 0
\end{align*}
in probability as $\alpha\rightarrow 2$.
\end{thm}
\begin{proof}
By \eqref{3-4.4}, we have
\begin{align*}
&\|e^{z(\omega_{\ell^\alpha})}\psi^{\alpha}(\omega_{\ell^\alpha}, e^{-z(\omega_{\ell^\alpha})}\xi)- e^{z(\omega)}\psi(\omega, e^{-z(\omega)}\xi)\|_{\sigma}\\
&\leq\|e^{z(\omega_{\ell^\alpha})}\psi^{\alpha}(\omega_{\ell^\alpha}, e^{-z(\omega_{\ell^\alpha})}\xi)- e^{z(\omega_{\ell^\alpha})}\psi^{\alpha}(\omega_{\ell^\alpha}, e^{-z(\omega)}\xi)\|_{\sigma}\\ \ \ &\ \ \ \ +\|e^{z(\omega_{\ell^\alpha})}\psi^{\alpha}(\omega_{\ell^\alpha}, e^{-z(\omega)}\xi)-e^{z(\omega)}\psi(\omega, e^{-z(\omega)}\xi)\|_{\sigma}\\
&\leq\frac{\mu}{2(1-\mu )}|e^{z(\omega_{\ell^\alpha})-z(\omega)}-1| \|\xi\|_{\sigma}\\
&\ \ \ +|e^{z(\omega_{\ell^\alpha})}|\|\psi^{\alpha}(\omega_{\ell^\alpha}, e^{-z(\omega)}\xi)-\psi(\omega, e^{-z(\omega)}\xi)\|_{\sigma}\\
&\ \ \ +|e^{z(\omega_{\ell^\alpha})}-e^{z(\omega)}|\|\psi(\omega, e^{-z(\omega)}\xi)\|_{\sigma}.
\end{align*}
From the same sub-subsequence argument, we obtain
\begin{align*}
\|e^{z(\omega_{\ell^\alpha})}\psi^{\alpha}(\omega_{\ell^\alpha}, e^{-z(\omega_{\ell^\alpha})}\xi)- e^{z(\omega)}\psi(\omega, e^{-z(\omega)}\xi)\|_{\sigma} \rightarrow 0
\end{align*}
in probability as $\alpha\rightarrow 2$.
Using the similar procedure, we have
\begin{align*}
\|D_{\xi}\big(e^{z(\omega_{\ell^\alpha})}\psi^{\alpha}(\omega_{\ell^\alpha}, e^{-z(\omega_{\ell^\alpha})}\xi)\big)- D_{\xi}\big(e^{z(\omega)}\psi(\omega, e^{-z(\omega)}\xi)\big)\|\rightarrow 0
\end{align*}
in probability as $\alpha\rightarrow 2$. Thus the proof is complete.
\end{proof}\par

\section*{Acknowledgement}
\indent
\par
The authors would like to thank the reviewers for their helpful comments. This work was partially supported by the
	National Natural Science Foundation of China(No.12326414,12271080), Sichuan Science and Technology Program(No.2023NSFSC0076).

\end{document}